\documentclass[12pt,a4paper]{article}
\usepackage{graphicx}
\usepackage{comment}
\usepackage{amssymb,amsmath, amsfonts,amsthm}
\usepackage[english]{babel}
\usepackage[utf8]{inputenc}
\usepackage{tikz}
\usepackage{pgfplots}
\usepackage{authblk}
\pgfplotsset{compat=1.18}
\usetikzlibrary{arrows.meta,positioning}
\pgfplotsset{compat=1.18}
\usepackage{listings}
\usepackage{relsize}

\usepackage{hyperref}
\usepackage{cancel}
\usepackage[normalem]{ulem}
\definecolor{mygray}{gray}{0.95}
\usepackage{float}
\usepackage{subcaption}
\usepackage{verbatim}
\input xy 
\xyoption{all} 

\usepackage{color}
\newtheorem{theorem}{Theorem}     
\newtheorem{definition}[theorem]{Definition}
\newtheorem{proposition}[theorem]{Proposition}
\newtheorem{lemma}[theorem]{Lemma}

\newtheorem{remark}[theorem]{Remark}

\newcommand{\eqa}{\begin{eqnarray}}
	\newcommand{\eeqa}{\end{eqnarray}}
\newcommand{\beq}{\begin{equation}}
	\newcommand{\eeq}{\end{equation}}

\makeatletter
\@addtoreset{equation}{section}

\makeatother

\begin{document}
\title{Contact Geometry of the Restricted Three-Body Problem on $\mathbb{S}^2$}
\author[1]{Alessandro Arsie}
\author[2]{K\"ur\c{s}at Y\i lmaz\footnote{corresponding author}}

\affil[1]{Department of Mathematics and Statistics,  The University of Toledo, Toledo,  OH,  USA, e-mail: alessandro.arsie@utoledo.edu}
\affil[2]{Department of Mathematics,  Washington and Lee University,  Lexington, VA,  USA, e-mail: kyilmaz@wlu.edu}
	
\date{}
	
	\maketitle
 Abstract: We study the contact geometry of the connected components of the energy hypersurface, in the symmetric restricted 3-body problem on $\mathbb{S}^2$, for a specific type of motion of the primaries. In particular, we show that these components are of contact type for all energies below the first critical value and slightly above it. At some critical steps, this is achieved using a computer-assisted proof in the form of validated numerics.
 
 We prove that these components, suitably compactified using a Moser-type regularization are contactomorphic to $\mathbb{RP}^3$ with its unique tight contact structure or to the connected sum of two copies of it, depending on the value of the energy. We exploit Taubes' solution of the Weinstein conjecture in dimension three, to infer the existence of periodic orbits in all these cases.
 \section{Introduction}
The history of the $N$-body problem on surfaces of constant nonzero curvature can be traced back to the 19th century through the works of Lobachevsky \cite{Lobachevsky} and Bolyai \cite{Boyai}. On a curved space, the gravitational attraction has to be modeled differently compared to what is usually done in a Euclidean setting, and there are in principle different reasonable choices for the modeling. One of the first explorers in this regard was Serret \cite{Serret}, who modeled the gravitational force, and consequently all the equations of motion, to $\mathbb{S}^2$, studying the corresponding Kepler problem. Much more work has been done in recent years on this topic. For a more detailed exploration of the $N$-body problem on curved spaces, we refer to \cite{Diacu1, Diacu2, Diacu3, Diacu4} and also \cite{Andrade} and \cite{Borisov}.

In contrast to the general problem, the restricted three-body problem consists of two primaries that attract each other according to gravitational forces in the ambient space, and their motions are known. The main focus is on the dynamics of a third body, which has considerably small mass and therefore does not have any significant impact on the motion of the two primaries.

The literature on the restricted three-body problem on \(\mathbb{S}^2\) is still rather limited,
especially when compared with the broader body of work on the \(N\)-body problem in
spaces of constant curvature. Most of the papers directly concerned with the restricted
problem on \(\mathbb{S}^2\) have focused on setting up the curved model, locating and analyzing
its equilibria, and understanding the local and perturbative dynamics. Early work of
Kilin considered libration points for the restricted problem on \(\mathbb{S}^2\),
showing in particular that positive curvature produces new equilibrium configurations
and a non-classical dependence on the parameters of the system (see \cite{Kilin}). Later, Mart\'inez
and Sim\'o studied relative equilibria of the restricted three-body problem in curved
spaces, including the spherical case, and described the corresponding bifurcations and
changes of spectral stability (see \cite{Martinez}). In the symmetric spherical setting most closely
related to the present paper, Andrade, P\'erez-Chavela and Vidal developed a Hamiltonian
formulation of the problem on surfaces of constant curvature and analyzed several
aspects of the global dynamics, proving in particular the existence of families of periodic
orbits and associated KAM \(2\)-tori (see \cite{Andrade}). See also Andrade--Vidal for a detailed study of
the polar equilibria on \(\mathbb{S}^2\), and Andrade--P\'erez-Chavela--Vidal for the regularization
of collision singularities in the circular problem, see \cite{ALM} and \cite{AndreadeVidal}. Thus, although some important
results are known, there is still not a large literature specifically devoted to the
restricted three-body problem on \(\mathbb{S}^2\), and in particular general existence results for
periodic orbits remain comparatively scarce. In contrast with the mainly dynamical,
local, or perturbative approaches above, the present paper uses tools from Symplectic
and Contact Topology, together with a Moser-type regularization, to deduce the
existence of periodic orbits from the contact-topological properties of the regularized
energy hypersurfaces.

In \cite{AL}, the authors started a program to use tools from Symplectic and Contact Topology to tackle some notoriously difficult issues in the analysis of the behavior of the restricted three-body problem. In particular, they showed that the energy hypersurfaces of the planar circular restricted three-body problem are three-dimensional non-compact manifolds and are all of contact type if the energy is below the first critical value or slightly above it. However, this is not sufficient to guarantee the existence of periodic orbits. To overcome this obstacle, they implemented a version of the Moser's regularization technique \cite{Moser}, to extend the dynamics to a closed manifold.

Our investigation deals with a restricted three-body problem on a surface of constant positive curvature $\kappa$, again using tools from Symplectic and Contact Topology to study the nature of the energy hypersurfaces and the existence of periodic orbits. Recall that any embedded surface with positive constant curvature $\kappa$ is a sphere $\mathbb{S}^2_r$ with radius $r=\sqrt{\frac{1}{\kappa}}$, embedded in $\mathbb{R}^3$. In this paper, we analyze the contact behavior of the energy hypersurfaces of the restricted three-body problem on such surfaces for the symmetric problem in which the masses of the primaries are equal. We focus only on the surface with Gaussian curvature $\kappa=1$; that is, we work on the unit sphere $\mathbb{S}^2$ embedded in $\mathbb{R}^3$. This can be extended to any positive value of $\kappa$ simply by rescaling. 

To simplify the equations, we normalize the masses so that the total mass is 1, which results in masses $m_1=m_2=1/2$. We consider a circular relative equilibrium, where the primaries rotate around each other at the same latitude with respect to the equator with a fixed angular velocity $\omega$ (see Figure \ref{fig:sphere}).

\begin{figure}[H]
	\begin{center}
		\IfFileExists{spherical.pdf}{\includegraphics[scale=1]{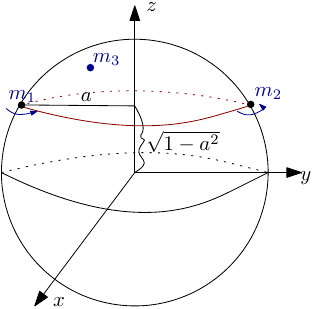}}{\begin{tikzpicture}[scale=1.0]\draw[thick] (0,0) circle (1.25);\draw[->] (-1.6,0)--(1.8,0) node[right] {$x$};\draw[->] (0,-1.5)--(0,1.7) node[above] {$z$};\fill (-0.55,0.55) circle (2pt) node[left] {$m_1$};\fill (0.55,0.55) circle (2pt) node[right] {$m_2$};\draw[dashed] (-0.9,0.55)--(0.9,0.55);\end{tikzpicture}}
		\caption{The symmetric restricted three-body problem on $ \mathbb{S}^2. $}
		\label{fig:sphere}
	\end{center}
\end{figure}

Due to the spherical metric and the presence of antipodal singularities, our Hamiltonian (see \eqref{H}) is substantially more complicated than the one analyzed in the planar case in \cite{AL}. We will not derive the Hamiltonian from scratch and instead refer the reader to \cite{ALM} for more details about its construction. We focus on the case where the motion of the primaries takes place on a circle with a radius of $ a=1/\sqrt{2} $ (and $ z=1/\sqrt{2} $); as shown in \cite{Andrade}, for this value of $ a $, there are only two critical points of the Hamiltonian, one located at the north pole and the other at the south pole. The north pole is the location in the configuration space of the first Lagrange point, which we denote by $ L_1 $ since the motion of the primaries takes place in the northern hemisphere.  We consider a regular value $ c $ of the Hamiltonian $H$ and examine the energy hypersurface $ H^{-1}(c) $ when $ c $ is below the value assumed by the Hamiltonian at the first Lagrange point and slightly above it. We note that if $ c<H(L_1) $, then the energy hypersurface $ H^{-1}(c) $ has two connected components that bound a region around the masses $ m_1 $ and $ m_2 $. We denote them by $ \Sigma_{c}^{m_1} $ and $ \Sigma_{c}^{m_2} $, respectively. If the value of $ c $ is slightly above $H(L_1)$, then the energy hypersurface has only one connected component, and we denote it by $ \Sigma_{c}$ (or to be safer, we consider as $\Sigma_{c}$ the relevant low-energy component containing both primaries, since in principle there might be other connected components around the antipodal singularities which we disregard, since the third body cannot reach them for levels of energy just above $H(L_1)$).

We restrict our analysis to this highly symmetrical configuration, i.e.\ the case
of equal masses moving along a rotating relative equilibrium, because the
two-body problem on a sphere is not integrable in general and for our approach
we need explicit rotating relative equilibria.  The geometry and dynamics of
the (unrestricted) two-body problem on $\mathbb{S}^2$ have been studied in detail by
Arsie and Balabanova: in
\cite{ArsieBalabanova-collisions} they analyse collision trajectories and
construct a regularisation via weighted blow-up, and in
\cite{ArsieBalabanova-geometry} they determine the topology of the
common level sets of the Hamiltonian and the Casimir and prove that, for
sufficiently negative energy, these level sets carry a global contact form
arising from a suitable Liouville vector field.  These results provide the
direct two-body antecedent of the contact-type analysis we carry out below for
the restricted three-body problem.

Other less symmetric rotating relative equilibria exist also for other ratios of the masses (see the very detailed \cite{Andrade} and \cite{Borisov2}). In principle, we expect similar results to hold also for a non-highly symmetric case as the one we are dealing with. 

All these relative rotating equilibria belong in general to families depending on parameters; it would be interesting to extend our results to the families themselves, but in that case the analytical estimates used here would have to be replaced completely by a treatment based exclusively on numerics. 

The contact-type inequalities obtained below have positive margins on the regularized compact pieces considered, and hence persist under sufficiently small perturbations of the chosen rotating relative equilibrium, for the corresponding nearby energy windows. We do not pursue a uniform parameter-family theorem here.

The structure of the paper is as follows. In Section 2, we introduce some preliminaries on contact structures. In Section 3 we verify the contact condition for the energy hypersurfaces $\Sigma_{c}^{m_1} $ (and $\Sigma_{c}^{m_2} $ respectively). We first switch to shifted polar coordinates $ (\rho,\theta) $ centered on one of the masses and show that the Liouville vector field $ X=\rho\partial_{\rho} $ is transverse to the component of the energy hypersurface under investigation. This is achieved using a computer-assisted proof, in particular using validated numerics implemented in Python with the Arb ball-arithmetic backend exposed by \texttt{python-flint} \cite{pythonflint,JohanssonArb}. 

In Section 4, we examine the Kepler's problem on a sphere. We apply a Moser-type regularization \cite{Moser} to complete the dynamics by first projecting the problem onto the plane. Then, we apply a symplectic transformation to interchange the roles of position and momenta and compactify the problem to a sphere so that the collision orbits pass through North Pole (of course, this last sphere is not the ``physical sphere" we started from). In Section 5, we extend the results given in Section 4 to the connected component of the energy hypersurfaces $ \Sigma_{c}^{m_1} $ (and $ \Sigma_{c}^{m_2} $ respectively) and show that these hypersurfaces can be regularized to form compact manifolds $ \widetilde{\Sigma}_c^{m_1} $ (and $ \widetilde{\Sigma}_c^{m_2} $ respectively). Moreover, we show that the regularized hypersurfaces are both diffeomorphic to $ \mathbb{R}P^3 $ with its unique tight contact structure. In Section 6, we prove that if $ c $ is slightly above the first critical value, the energy hypersurface $ \Sigma_c $ is also of contact type and can be regularized to $\widetilde{\Sigma}$ which is diffeomorphic to $ \mathbb{R}P^3\# \mathbb{R}P^3 $, the connected sum (in the contact category) of two three-dimensional real projective spaces. The existence of periodic orbits will be deduced from all of this.

{\bf Acknowledgments}

The authors would like to thank Augustin Moreno, Natasha Balabanova and Aaron Brown for fruitful discussions. Furthermore, the authors would like to thank the referees for comments and suggestions which have helped improve the quality of the paper.  We want to thank also Siegfried Rump for creating the package IntLab that was used for validated numerics in MatLab in a previous version of the paper. 

{\bf Validated Numerics}

To estimate some quantities, in particular to prove transversality of a Liouville vector field with respect to a hypersurface, we use validated numerics.  The computer-assisted checks are implemented in Python using \texttt{python-flint}, which provides Arb ball arithmetic with rigorous outward rounding \cite{pythonflint,JohanssonArb}.  The verifier uses interval evaluation of the displayed formulae, interval automatic differentiation for the derivative quantities, and adaptive bisection of terminal boxes until the required interval inequality is certified.  The full middle annulus is certified directly on the Hill region, including all boundary strips.  The supplied execution log documents Arb-precision runs at 200 and 256 bits; every certificate passes in both runs. The quoted endpoint values are taken from the 256-bit run.
\footnote{The complete implementation is provided as supplementary
material in the ancillary file
\texttt{rigorous\_validated\_numerics.py}, which we recommend
invoking at 200 or 256 bits of Arb precision; the supplied
execution log records both runs and every certificate passes in
both. The endpoint values quoted in the body of the paper are
taken from the 256-bit run. Several certificates have small
absolute margins above their stated thresholds and are not
certifiable with standard 53-bit (IEEE-double) interval
arithmetic. In particular: the far-strip step in the proof of
Claim~A inside Theorem~\ref{thm: principal} certifies
$\widehat U(\rho,x)-\widehat U(\rho,1)>10^{-4}$ with worst lower
endpoint $\approx 10^{-4}+8\cdot 10^{-11}$; the middle-annulus
check $F_{\mathrm{mid}}>0.078$ has worst lower endpoint
$\approx 0.078\,000\,16$; the boundary check
$\widehat U_a(x)+1>0.037$ in Lemma~\ref{lambda} has worst lower
endpoint $\approx 0.037\,003\,07$ (above the threshold by
$\approx 3\cdot 10^{-6}$); and the inner-annulus check
$B_0>4\cdot 10^{-3}$ in Claim~C has worst lower endpoint
$\approx 0.004\,006\,9$. These widths are produced by adaptive
bisection against the fixed thresholds at the working precision;
dense scalar sampling shows that the underlying functions sit
substantially further above their thresholds than the certificate
widths suggest. Appendix~\ref{app:numerics} tabulates every
certificate with margins.}

\section{Preliminaries on Contact Structures}\label{section2}
In this short section, we summarize some well-known results about Contact Geometry and Contact Topology that will be used in the paper. In particular, we will recall how results in Contact/Symplectic Topology provide existence of periodic orbits of a Hamiltonian system. For more on this, see \cite{Ottobook} and the illuminating survey \cite{moreno2022contact}.

\begin{definition} Let $X$ be a manifold of dimension $2n-1$.
A contact form $\alpha$ on $X$ is a 1-form such that $\alpha\wedge (d\alpha)^{n-1}$
is a nowhere vanishing volume form on $X$. $(X, \alpha)$ is called a (co-oriented) contact manifold. 
\end{definition}

Often one is interested in a weaker notion, that of a {\em contact distribution} or contact structure on $X$: this is just a hyperplane distribution $\eta\subset TX$ which is maximally non-integrable (locally $\eta=\ker(\alpha)$ where $\alpha$ is a (local) 1-form and the condition $\alpha\wedge (d\alpha)^{n-1}\neq 0$ means that $\eta$ is maximally non-integrable).  Notice that if $\alpha$ is a contact form, then $d\alpha_{|\eta}$ is symplectic on $\eta$. For our purposes, we will always use the stronger notion of co-oriented contact manifold $(X, \alpha),$ since this is the notion of contact manifold related to dynamics. 

Indeed, to $(X, \alpha)$ there is an associated canonical vector field, called the Reeb vector field $R_{\alpha}$ which is defined via the two conditions $i_{R_{\alpha}}d\alpha=0$ and $i_{R_{\alpha}}\alpha=1$.
It turns out that the Reeb vector field is always transverse to the contact distribution $\eta:=\ker(\alpha)$ induced by $\alpha.$ This is because $d\alpha_{|\eta}$ is symplectic. 

In many cases, the Reeb vector field is just a (positive) time reparametrization of a Hamiltonian vector field. For this to occur, let $(M, \omega)$ be a symplectic manifold, let $H:M \rightarrow \mathbb{R}$ be a smooth Hamiltonian and let $X:=H^{-1}(e)$, where $e$ is a regular value of $H.$ Assume that there exists a tubular neighborhood $U$ of $X$ in $M$ such that $\omega_{|U}$ is exact, say equal to $d\lambda.$ Thus on $U$ it makes sense to look for a vector field $V$ such that $i_V\omega=\lambda.$ Any such a vector field is called a {\em Liouville} vector field. In particular, if $V$ is a Liouville vector field, it follows automatically that $L_V\omega=\omega$, where $L_V$ is the Lie derivative along $V$. Thus $V$ acts as a symplectic dilation on $U$.

Then the following holds:
\begin{theorem}\label{cont.th.1}
In the situation above, if $V$ is transverse to $X$, then $i_V\omega_{|X}=\lambda_{|X}$ is a contact form $\alpha$ on $X$. Furthermore, the Reeb vector field $R_{\alpha}$ associated to this contact form is a (positive) time reparametrization of the Hamiltonian vector field $X_H$.
\end{theorem}
For a proof one can check section 2.5 of the book \cite{Ottobook}. Part of the importance of this result stems from the fact that the Hamiltonian dynamics is described in a completely geometric manner. All the results about the dynamics of the Reeb vector field (like existence of periodic orbits) can be translated into results about Hamiltonian dynamics. A typical result in this direction is:

\begin{theorem}[Taubes, \cite{Taubes}]\label{thm: taubes}If $M$ is a closed oriented three-manifold with a contact form $\alpha$, then the associated Reeb vector field has a closed orbit.
\end{theorem}
That is exactly the Weinstein conjecture in dimension 3 (see \cite{Weinstein1979}). Long before Taubes proved the conjecture in dimension 3, Viterbo gave a proof for a special case in which the contact manifolds considered are compact energy hypersurfaces of contact type in $\mathbb{R}^{2n}$ equipped with the standard symplectic form (see \cite{viterbo}). What is meant by being of contact type is as follows:

\begin{definition}
    A hypersurface $\Sigma$ in a symplectic manifold $(M,\omega)$ is of contact type if there exists a global 1-form $\alpha$ which defines the contact structure on $\Sigma$ such that $d\alpha=\omega|_\Sigma$ and $\alpha(v)\neq 0  $ for all $ \ v \in L_\Sigma:=Ker(\omega|_\Sigma ) \subset T\Sigma$.
\end{definition}

The above definition is equivalent to the existence of a Liouville vector field which is transverse everywhere to $\Sigma$. For further detail one can refer to \cite{moreno2022contact}.

There are two types of contact structures.
\begin{definition}
     A contact structure is called overtwisted if there exists an embedded disk that is tangent to the contact plane distribution along its boundary. If there is no such disk then the structure is called tight. 
\end{definition}
One way to show tightness of a contact structure is through fillability. There are several notions for fillability but we will mainly use the following: 

\begin{definition}
    Let $(M,\xi)$ be a contact manifold (here $\xi$ denotes the contact distribution or contact structure, that is a maximally non-integrable hyperplane distribution). A compact, connected symplectic manifold \((W, \omega)\) with boundary \(\partial W = M\) is a weak filling of \((M, \xi)\) if \(\omega|_{\xi} > 0\). It is a strong filling if \(\xi = \ker(\iota_Y \omega)\) for some vector field \(Y\) defined near \(\partial W\), which points transversely outward at the boundary and satisfies \(L_Y \omega = \omega\). If \(Y\) extends globally over \(W\), then \(\iota_Y \omega\) defines a global primitive of \(\omega\), making \((W, \omega)\) an exact filling.
    
\end{definition}
There are numerous results regarding fillability; for more details, see \cite{wendl2010}.

In our paper, only 2-body collisions are dynamically possible. In general, due to the presence of 2-body collisions, energy hypersurfaces are not compact. In this case, Taubes' result cannot be directly applied even if the hypersurfaces involved are of contact type. However, the hypersurfaces can be compactified by using a Moser-type regularization which preserves the dynamics on the ambient manifold.

When regularized, the energy hypersurfaces for the planar restricted $3$-body problem at energies below a certain threshold \cite{AL} are diffeomorphic to $\mathbb{R}P^3$. Despite our Hamiltonian being more complicated, when the energy is low enough we deduce the same result. Moreover, we see that the regularized energy hypersurfaces carry a unique tight contact structure. This is achieved using the following result.
\subsection{Weinstein model for connected sum}
In this subsection a brief model for constructing connected sums of contact manifolds will be given.

Consider \((\mathbb{R}^{2n+2}, \omega_0 = d\tilde{x} \wedge d\tilde{y} + dz \wedge dw)\) with its standard symplectic structure, where \(\tilde{x}, \tilde{y} \in \mathbb{R}^n\) and \(z, w \in \mathbb{R}\). Notice that the Liouville vector field
\[
X = \frac{1}{2} (\tilde{x} \, \partial_{\tilde{x}} + \tilde{y} \, \partial_{\tilde{y}}) + 2z \, \partial_z - w \, \partial_w
\]
is transverse to the level sets of \(f = \tilde{x}^2 + \tilde{y}^2 + z^2 - w^2\) whenever \(f \neq 0\). This observation is essential in demonstrating that connected sums of contact manifolds carry a contact structure.

Let \((M_1, \xi_1)\) and \((M_2, \xi_2)\) be contact manifolds, and select Darboux balls \(D_i \subset M_i\). For the contact connected sum, note that we can embed \(D_1 \cup D_2\) in the two connected components of \(f^{-1}(-1)\). In fact, \(f^{-1}(-1)\) can be viewed as two separate Darboux balls embedded as contact hypersurfaces in \(\mathbb{R}^{2n+2}\). Next, we join these two Darboux balls by ``interpolating" \(f^{-1}(-1)\) to \(f^{-1}(1)\), forming the ``neck" of the connected sum. This interpolation is feasible because \(f^{-1}(-1)\) and \(f^{-1}(1)\) are near each other when \(w\) is large. The Liouville vector field \(X\) remains transverse to this neck, ensuring that the connected sum inherits a contact structure.

We are using this model on a tight contact manifold and hence basically we are doing surgery of index one which leads again to a tight contact structure. For more details see \cite{colin}.

\begin{theorem}[Eliashberg, \cite{Eliashberg}]\label{thm: Eliashberg}
    The closed 3-manifold $\mathbb{R}P^3$ carries a unique tight contact structure up to isotopy.
\end{theorem}
This unique tight contact structure is, in fact, symplectically fillable; its filling is $D^*\mathbb{S}^2,$ the cotangent closed unit disk in $T^*\mathbb{S}^2$. The following  decomposition theorem will play a role in Section 6.
\begin{theorem}[Ding-Geiges, \cite{ding}] \label{thn: ding-geiges}Every non-trivial tight contact 3-manifold \((M, \xi)\) is contactomorphic
to a connected sum
\[
(M_1, \xi_1) \# \dots \# (M_k, \xi_k)
\]
of finitely many prime tight contact 3-manifolds. The summands are unique up to order and contactomorphism.

\end{theorem}
	\section{The Hamiltonian for the Problem}

	We start considering a circular relative equilibrium of the spherical 2-body problem (the two primaries), in which the bodies move along circles with fixed heights. For convenience, we normalize the masses of the primaries to $ m_1=m_2=1/2 $, as we consider the symmetric case, and we let $ \boldsymbol{q}_i=(x_i,y_i,z_i)  \in \mathbb{R}^3 $ such that $ |\boldsymbol{q}_i|=1 $ represent the positions of the primaries for $ i=1,2 $ on the unit sphere embedded in $\mathbb{R}^3$.  According to \cite{nbodyp1} the equations of the motion of the primaries are described by the following equations:
	
	\begin{equation}\label{eqmotion}
	\begin{split}
	\boldsymbol{\ddot{q}}_1&=-(\boldsymbol{\dot{q}}_1\cdot \boldsymbol{\dot{q}}_1)\boldsymbol{q}_1+\dfrac{\boldsymbol{q}_2-(\boldsymbol{q}_1\cdot \boldsymbol{q}_2)\boldsymbol{q}_1}{2[1-(\boldsymbol{q}_1\cdot\boldsymbol{q}_2)^2]^{3/2}},\\
	\boldsymbol{\ddot{q}}_2
 &=-(\boldsymbol{\dot{q}}_2\cdot \boldsymbol{\dot{q}}_2)\boldsymbol{q}_2+\dfrac{\boldsymbol{q}_1-(\boldsymbol{q}_1\cdot \boldsymbol{q}_2)\boldsymbol{q}_2}{2[1-(\boldsymbol{q}_1\cdot\boldsymbol{q}_2)^2]^{3/2}},
	\end{split}\end{equation}
	with the restrictions 
	\begin{equation}
	\boldsymbol{q}_i\cdot \boldsymbol{q}_i=1, \ \ \ \boldsymbol{q}_i\cdot \boldsymbol{\dot{q}}_i=0, \ \ \ i=1,2,
	\end{equation}
	and where $ -\cdot- $ is the standard scalar product in $ \mathbb{R}^3 $.
	
	A circular relative equilibrium \cite{Andrade} is a solution of \eqref{eqmotion} of the following form (see Figure \ref{fig:sphere} for an illustration)
	\begin{equation}
	\boldsymbol{q}_i=(r_i\cos(\omega t+a_i),r_i\sin(\omega t+a_i),z_i),
	\end{equation}
	where $ z_i \in (-1,1) $, $ a_i \in [0,2\pi] $, $\omega\neq 0$ and $ r_i=(-1)^i\sqrt{1-z_i^2} $.
	
	Since we are considering the two primaries moving on the same circle we fix $|r_1|=|r_2|=r>0$. Because of the relation between the angular momentum  and the radius given by $\omega^2=\Big(\dfrac{1}{4r^2-4r^4}\Big)^{3/2}$, fixing the angular momentum $\omega=1$, the radius of motion is fixed to be $r=1/\sqrt{2}$. We also fix the phases $ a_i=0$, $i=1,2$. For this configuration, there are only two critical points of the Hamiltonian function for the system (see Table 1 in \cite{Andrade}), one is located at north pole and the other at south pole of the sphere. The motion of the primaries takes place on the northern hemisphere (indeed $ z_i=1/\sqrt{2} $). We apply a stereographic projection from the south pole to the equatorial plane ($ xy $-plane) and then pass to a rotating coordinate system in which the position of the primaries are fixed at $ \big(\pm (\sqrt{2}-1),0\big) $. We denote the coordinates of the third mass in the rotating system by $ \vec{q}=(q_1,q_2) $.

 In this way we get an autonomous Hamiltonian $ H:T^\ast(\mathbb{R}^2\setminus\{(\pm(\sqrt{2}-1),0)\} \to \mathbb{R}$ which represents the total energy of the third mass (satellite) in the rotating coordinates where the primaries are fixed. This is given as:
	$$H=H_1+H_2+U_1+U_2$$
	where the different terms have the following expressions:
	\begin{equation}\label{H}
		\begin{split}
			H_1=&\dfrac{(q_1^2+q_2^2+1)^2(p_1^2+p_2^2)}{8},\\
			H_2=&q_2p_1-q_1p_2,\\
			U_1=&-\dfrac{\sqrt{2}q_1-\dfrac{\sqrt{2}(q_1^2+q_2^2-1)}{2}}{2\sqrt{(q_1^2+q_2^2+1)^2-\Big(\sqrt{2}q_1-\dfrac{\sqrt{2}(q_1^2+q_2^2-1)}{2}\Big)^2}},\\
			U_2=&-\dfrac{-\sqrt{2}q_1-\dfrac{\sqrt{2}(q_1^2+q_2^2-1)}{2}}{2\sqrt{(q_1^2+q_2^2+1)^2-\Big(-\sqrt{2}q_1-\dfrac{\sqrt{2}(q_1^2+q_2^2-1)}{2}\Big)^2}}.
		\end{split}
	\end{equation}
	
	We combine the two terms $H_1$ and $H_2$ to complete the square and produce an effective potential and we write the new Hamiltonian as
	$H=K+U_0+U_1+U_2,$
	where $K$ is the kinetic energy for the new Hamiltonian, $U_0$ is the additional term of the potential that comes from completing the square, $U_1$ is the potential generated by mass $m_2$ and $U_2$ is the potential generated by mass $m_1$. Explicitly:
	
	\begin{equation}\label{Kqp}
		K=\dfrac{\Big((q_1^2+q_2^2+1)p_1+\dfrac{4q_2}{q_1^2+q_2^2+1}\Big)^2}{8}+\dfrac{\Big((q_1^2+q_2^2+1)p_2-\dfrac{4q_1}{q_1^2+q_2^2+1}\Big)^2}{8},
	\end{equation}
	\begin{equation}\label{Uqp}
		\begin{split}
			U_0=&-\dfrac{2q_1^2+2q_2^2}{(q_1^2+q_2^2+1)^2},\\
			U_1=&-\dfrac{\sqrt{2}q_1-\dfrac{\sqrt{2}(q_1^2+q_2^2-1)}{2}}{2\sqrt{(q_1^2+q_2^2+1)^2-\Big(\sqrt{2}q_1-\dfrac{\sqrt{2}(q_1^2+q_2^2-1)}{2}\Big)^2}},\\
			U_2=&-\dfrac{-\sqrt{2}q_1-\dfrac{\sqrt{2}(q_1^2+q_2^2-1)}{2}}{2\sqrt{(q_1^2+q_2^2+1)^2-\Big(-\sqrt{2}q_1-\dfrac{\sqrt{2}(q_1^2+q_2^2-1)}{2}\Big)^2}}.
		\end{split}
	\end{equation}
	
	We start with a preliminary result:
	
	\begin{theorem}\label{Theorem_1}
		In polar coordinates centered at the first mass $(m_1),$ the Liouville vector field $ X=\rho\partial_{\rho} $ is transverse to the connected component of the energy hypersurface $ H^{-1}(c) $ containing the first mass $ m_1 $, $ \Sigma_{c}^{m_1} $, for sufficiently small energies. Thus the corresponding components of the energy hypersurface are of contact type.
	\end{theorem}

	\begin{proof}
	For convenience, we pass to polar coordinates $(\rho, \theta)$ centered at one of the primaries, in our case it will be $\vec{q}_{m_1}$ located at $(-(\sqrt{2}-1),0)$. It is easy to check that $(\vec{q}-\vec{q}_{m_1})\nabla_q$ is a Liouville vector field and preserves the standard symplectic form and so does its transformation to polar coordinates given by $ X=\rho\partial_{\rho} $. Moreover, in polar coordinates, the potential $ U(\rho,\theta) $ becomes:
	$$U(\rho,\theta)=\mathcal{U}_0(\rho,\theta)+\mathcal{U}_1(\rho,\theta)+\mathcal{U}_2(\rho,\theta),$$ where $ \mathcal{U}_i(\rho,\theta) $'s are as below:
	
	Observe that $ q_1^2+q_2^2 $ in polar coordinates is given by $\Big(\rho\cos(\theta)-(\sqrt{2}-1)\Big)^2+\rho^2\sin^2(\theta):=A$. Then we have:
	
\begin{equation}\label{U_i}
\begin{split}
\mathcal{U}_0(\rho,\theta)=&-
\dfrac{2\rho^2\sin^2\theta+2\bigl(\rho\cos\theta-(\sqrt2-1)\bigr)^2}{(A+1)^2},\\
\mathcal{U}_1(\rho,\theta)=&-
\dfrac{\sqrt2\bigl(\rho\cos\theta-(\sqrt2-1)\bigr)-\dfrac{\sqrt2}{2}(A-1)}
{2\sqrt{(A+1)^2-\bigl(\sqrt2\bigl(\rho\cos\theta-(\sqrt2-1)\bigr)-\dfrac{\sqrt2}{2}(A-1)\bigr)^2}},\\
\mathcal{U}_2(\rho,\theta)=&-
\dfrac{-\sqrt2\bigl(\rho\cos\theta-(\sqrt2-1)\bigr)-\dfrac{\sqrt2}{2}(A-1)}
{2\sqrt{(A+1)^2-\bigl(-\sqrt2\bigl(\rho\cos\theta-(\sqrt2-1)\bigr)-\dfrac{\sqrt2}{2}(A-1)\bigr)^2}}.
\end{split}
\end{equation}

	It is immediate to check that $\lim_{\rho\rightarrow 0^{+}}U=-\infty$ and all the summands of $U$ except third are finite when $\rho\rightarrow 0^{+}$.
	
	Let $\Sigma_{c}^{m_1}$ denotes the connected component of the energy hypersurface $H^{-1}(c)$ containing $m_1$ in $T^*\mathbb{R}^2$. Call $\Lambda_c^{m_1}$ the corresponding Hill's region which is defined as the region of allowed motions, i.e,. the projection of the hypersurface $\Sigma_{c}^{m_1}$ onto the $q$-plane. Alternatively $$\Lambda_c^{m_1}:=\text{connected component of } \{q \in \mathbb{R}^2:U(q)\leq c\} \  \text{containing $ m_1 $}.$$
	We need to show that $X(H)_{\Sigma_{c}^{m_1}}\neq 0$ if $c<<0$. 
	Write $K=\frac{1}{8}(f_1^2+f_2^2).$ Then $X(K)=\frac{\rho}{4}(f_1g_1+f_2g_2)$ where $g_i=\partial_{\rho}f_i$. By the Cauchy-Schwartz inequality we have that $(f_1g_1+f_2g_2)^2\leq (f_1^2+f_2^2)(g_1^2+g_2^2)$. Therefore:
	\begin{equation}\label{|xk|}
	|X(K)|\leq \frac{\rho}{4}\sqrt{f_1^2+f_2^2}\sqrt{g_1^2+g_2^2}=\frac{\rho}{4}\sqrt{8K}\sqrt{g_1^2+g_2^2}.
	\end{equation}
	
	When evaluated in polar coordinate $ f_i$'s and $ g_i $'s are written explicitly as follows:
	\begin{equation}\label{fi}
		\begin{split}
			f_1(\rho,\theta)=&a_1(\rho,\theta)p_1+b_1(\rho,\theta),\ \ \ \ \ \ f_2(\rho,\theta)=a_2(\rho,\theta)p_2+b_2(\rho,\theta),\\
			g_1(\rho,\theta)=&c_1(\rho,\theta)p_1+d_1(\rho,\theta), \ \ \ \ \ \ g_2(\rho,\theta)=c_2(\rho,\theta)p_2+d_2(\rho,\theta),
		\end{split}
	\end{equation}
	
	where $ a_i, b_i,c_i$ and $ d_i $ for $ i=1,2 $ are as follows:
	\begin{equation}\label{a_i}
		\begin{split}
			a_1:=&\Big(\Big(\rho\cos(\theta)-(\sqrt{2}-1)\Big)^2+\rho^2\sin^2(\theta)+1\Big),\\
			a_2:=&\Big(\Big(\rho\cos(\theta)-(\sqrt{2}-1)\Big)^2+\rho^2\sin^2(\theta)+1\Big),\\
			b_1:=&\dfrac{4\rho\sin(\theta)}{\Big(\rho\cos(\theta)-(\sqrt{2}-1)\Big)^2+\rho^2\sin^2(\theta)+1},\\
			b_2:=-&\dfrac{4\Big(\rho\cos(\theta)-(\sqrt{2}-1)\Big)}{\Big(\rho\cos(\theta)-(\sqrt{2}-1)\Big)^2+\rho^2\sin^2(\theta)+1},\\
		c_1:=&\Big(2\Big(\rho\cos(\theta)-(\sqrt{2}-1)\Big)\cos(\theta)+2\rho\sin^2(\theta)\Big),\\
			c_2:=&\Big(2\Big(\rho\cos(\theta)-(\sqrt{2}-1)\Big)\cos(\theta)+2\rho\sin^2(\theta)\Big),\\ d_1:=&\dfrac{-4\sin(\theta)(\rho^2-4+2\sqrt{2})}{\Bigg(\Big(\rho\cos(\theta)-(\sqrt{2}-1)\Big)^2+\rho^2\sin^2(\theta)+1\Bigg)^2},\\
			d_2:=&\dfrac{4\cos(\theta)(\rho^2+2-2\sqrt{2})+8\rho(1-\sqrt{2})}{\Bigg(\Big(\rho\cos(\theta)-(\sqrt{2}-1)\Big)^2+\rho^2\sin^2(\theta)+1\Bigg)^2}
		\end{split}
	\end{equation}

\begin{lemma}\label{alpha_beta}
		For every finite number $\rho^\ast>0$ there exist positive constants
		$\alpha,\beta$ such that
		\[
			g_1^2+g_2^2\leq \alpha K+\beta
		\]
		for all $(\rho,\theta)\in[0,\rho^\ast]\times S^1$ and all
		$(p_1,p_2)\in\mathbb R^2$, where $K$, $g_1$, and $g_2$ are defined above.
	\end{lemma}
	\begin{proof}
		Put
		\[
			\mathcal S:=[0,\rho^\ast]\times S^1.
		\]
		For $i=1,2$ we have
		\[
			f_i=a_i(\rho,\theta)p_i+b_i(\rho,\theta),\qquad
			g_i=c_i(\rho,\theta)p_i+d_i(\rho,\theta).
		\]
		By the explicit formulae in \eqref{a_i},
		\[
			a_i(\rho,\theta)=
			\Big(\rho\cos\theta-(\sqrt2-1)\Big)^2+\rho^2\sin^2\theta+1\geq 1
		\]
		on $\mathcal S$. Hence $1/a_i$ is continuous and bounded on $\mathcal S$.
		All the coefficient functions $a_i,b_i,c_i,d_i$ are continuous on $\mathcal S$;
		therefore the functions
		\[
			\Gamma_i(\rho,\theta):=\frac{c_i(\rho,\theta)}{a_i(\rho,\theta)},
			\qquad
			\Delta_i(\rho,\theta):=d_i(\rho,\theta)-
			\frac{c_i(\rho,\theta)b_i(\rho,\theta)}{a_i(\rho,\theta)}
		\]
		are continuous on the compact set $\mathcal S$. Let
		\[
			C_i:=\max_{\mathcal S}|\Gamma_i|,
			\qquad
			D_i:=\max_{\mathcal S}|\Delta_i|.
		\]
		These numbers are finite.

		Solving the relation defining $f_i$ for $p_i$ gives
		\[
			p_i=\frac{f_i-b_i}{a_i}.
		\]
		Substituting this into $g_i$ yields the exact identity
		\[
			g_i=\frac{c_i}{a_i}f_i+
			\left(d_i-\frac{c_i b_i}{a_i}\right)
			=\Gamma_i f_i+\Delta_i.
		\]
		Using $(u+v)^2\leq 2u^2+2v^2$, we obtain, uniformly on $\mathcal S$ and for
		all $p_i\in\mathbb R$,
		\[
			g_i^2\leq 2C_i^2 f_i^2+2D_i^2.
		\]
		Choose, for example,
		\[
			\alpha:=16\max\{C_1^2,C_2^2,1\},
			\qquad
			\beta:=1+2D_1^2+2D_2^2.
		\]
		Then $\alpha>0$, $\beta>0$, and
		\[
			2C_i^2 f_i^2\leq \frac{\alpha}{8}f_i^2,
			\qquad i=1,2.
		\]
		Consequently
		\[
			g_1^2+g_2^2
			\leq \frac{\alpha}{8}(f_1^2+f_2^2)+\beta
			=\alpha K+\beta,
		\]
		because $K=\frac18(f_1^2+f_2^2)$. This proves the claim.
	\end{proof}

	By using Lemma \ref{alpha_beta}, inequality \eqref{|xk|} becomes:
	
	$$|X(K)|\leq \frac{\rho}{4}\sqrt{8K}\sqrt{\alpha K +\beta}.$$
	Therefore 
	\begin{equation}\label{X(H)}
		X(H)=X(K)+X(U)\geq X(U)-\frac{\rho}{4}\sqrt{8K}\sqrt{\alpha K +\beta}.
	\end{equation}
	Thus 
	\begin{equation}\label{imp.eq}X(H)_{\Sigma_{c}^{m_1}}\geq X(U)-\frac{\rho}{4}\sqrt{8K}\sqrt{\alpha K +\beta}|_{\Sigma_{c}^{m_1}}=\rho\partial_{\rho}U-\frac{\rho}{4}\sqrt{8(c-U)}\sqrt{\alpha (c-U) +\beta}|_{\Lambda_c^{m_1}}.\end{equation}

	\begin{lemma}\label{limrhou}
		Let $\mathcal{U}_2 $ be the last summand of $ U $ in polar coordinates given in \eqref{U_i}. The limit $\lim_{\rho\rightarrow 0^{+}}\rho \ \mathcal{U}_2$ is finite. Furthermore,  $$\lim_{\rho\rightarrow 0^{+}}\frac{\rho}{4}\sqrt{8(c-U)}\sqrt{\alpha (c-U) +\beta}$$ exists and is finite equal to 
        $$\dfrac{\sqrt{2}-1}{2}\sqrt{\alpha} .$$
	\end{lemma}
	\begin{proof} This is a straightforward computation.

	\end{proof}

	\begin{lemma}\label{limrhodu}
		We have	$$\lim_{\rho\rightarrow 0^{+}}\rho\partial_{\rho}U-\frac{\rho}{4}\sqrt{8(c-U)}\sqrt{\alpha (c-U) +\beta}=+\infty,$$ uniformly in $\theta$.
	\end{lemma}
	\begin{proof}
		By Lemma \ref{limrhou} $\lim_{\rho\rightarrow 0^{+}}\frac{\rho}{4}\sqrt{8(c-U)}\sqrt{\alpha (c-U) +\beta}$ exists and is constant. So our main focus will be $ \lim_{\rho\rightarrow 0^{+}}\rho\partial_{\rho}U $. As we did in the previous lemma we consider $ U $ as the sum of the three components $ U=\mathcal{U}_0+\mathcal{U}_1+\mathcal{U}_2 $ where $ \mathcal{U}_i $'s are as given in the equation \ref{U_i}. When we check the Laurent series centered at $\rho=0$ for $\rho\partial_{\rho}\mathcal{U}_0$, the first term with nonzero coefficient is $$\Big(\frac{(\sqrt{2}-1)^2\cos{\theta}}{(2-\sqrt{2})^3}\Big)\rho,$$ and for $\rho\partial_{\rho}\mathcal{U}_1$, we have $$\Big(\frac{(-3+2\sqrt{2})\cos{\theta}}{(2-\sqrt{2})^3}\Big)\rho,$$ and hence all the limits $ \rho \partial_{\rho}U_i $ are $ 0 $. Now we check $ \rho\partial_{\rho} \mathcal{U}_2 $. When we checked the Laurent series for $ \rho\partial_{\rho} \mathcal{U}_2 $ the first term with nonzero coefficient we get is $$\dfrac{2-\sqrt{2}}{2}\rho^{-1},$$ which tends to $ +\infty $ as $\rho\rightarrow 0^{+}$. Thus, the limit $$\lim_{\rho\rightarrow 0^{+}}\rho\partial_{\rho}U-\frac{\rho}{4}\sqrt{8(c-U)}\sqrt{\alpha (c-U) +\beta}=+\infty.$$
	\end{proof}

	{\bf Conclusion: }The proof of the Theorem \ref{Theorem_1} follows directly from the Lemmas \ref{alpha_beta}, \ref{limrhou}, \ref{limrhodu} and the inequality (\ref{X(H)}) as we have shown the existence of a Liouville vector field which is transverse everywhere to $ H^{-1}(c) $. That indicates that for all $c<<0$ such that $\Lambda_c^{m_1}\subset B(0, \rho^*)$ the energy hypersurfaces are of contact type. 
	\end{proof}
	
	The previous theorem is of asymptotic value, in the sense that it proves the existence of energy hypersurfaces of contact type only for very negative energies.
	
	Now we start analyzing the situation for values of energy below the energy of the first Lagrange point. The energy corresponding to first Lagrange point (located at the north pole of the sphere) is $-1 $.

\begin{lemma}\label{lambda}
	Let $a:=\sqrt{2}-1$, and let
	\[
	\mathbb D_a:=\{(\rho,\theta)\mid 0\leq \rho\leq a\}
	\]
	be the closed disk in the $q$-plane centered at $q_{m_1}$.  For every $c<-1$, the component $\Lambda_c^{m_1}$ of $\{U\leq c\}$ whose closure contains the collision point $q_{m_1}$ is contained in the open disk $\{0<\rho<a\}$.  In particular, $\Lambda_c^{m_1}$ is properly contained in the closed disk centered at $q_{m_1}$ with radius $\sqrt{2}-1$.
\end{lemma}

\begin{proof}
	Fix $c<-1$.  Since $U(\rho,\theta)\to -\infty$ as $\rho\to0^+$ uniformly in $\theta$, there exists $r_c\in(0,a)$ such that
	\[
	U(\rho,\theta)<c \qquad \text{for all }0<\rho<r_c \text{ and all }\theta.
	\]
	Thus the sublevel set $\{U\leq c\}$ has a distinguished component accumulating at the collision point $q_{m_1}$; this is the component denoted by $\Lambda_c^{m_1}$.
	
	It remains to prove that this component cannot meet the boundary circle $\partial\mathbb D_a=\{\rho=a\}$.  Put $x:=\cos\theta$ and denote
	\[
	\widehat U_a(x):=U(a,\theta), \qquad -1\leq x\leq 1.
	\]
	A direct simplification of \eqref{U_i} gives the following one-variable expression:
	\begin{equation}\label{Ua_boundary_formula}
	\begin{split}
	\widehat U_a(x)=&-
	\frac{4a^2(1-x)}{\bigl(1+2a^2(1-x)\bigr)^2}
	+
	\frac{a-2\sqrt2\,x}{2\sqrt{(a^2+4ax+4)(5-4x)}} \\
	&+
	\frac{\frac{a}{2}+(2-\sqrt2)x-2}{\sqrt{a^2-4\sqrt2\,a x+8}} .
	\end{split}
	\end{equation}
	The square-root arguments in \eqref{Ua_boundary_formula} are positive for all $x\in[-1,1]$.  Moreover,
	\[
	\widehat U_a(1)=0-\frac12-\frac12=-1.
	\]
		We now verify that no smaller value occurs on the rest of the boundary.  This is a one-dimensional interval computation, independent of the momentum estimates used later.  The Arb verifier adaptively bisects the interval and proves
		\[
		\widehat U_a(x)+1>0.037 \qquad\text{for all }x\in[-1,9/10],
		\]
		with worst certified lower endpoint greater than \(0.0370030681813\).  It also proves
		\[
		\frac{d}{dx}\widehat U_a(x)<-0.20 \qquad\text{for all }x\in[9/10,1],
		\]
		with worst certified upper endpoint less than \(-0.2175374251559\).
	Thus \(\widehat U_a\) is strictly decreasing on \([9/10,1]\), and since the endpoint is evaluated exactly as \(\widehat U_a(1)=-1\), we get
	\[
	\widehat U_a(x)\geq \widehat U_a(1)=-1 \qquad\text{for }x\in[9/10,1].
	\]
	Together with the interval lower bound on \([-1,9/10]\), this proves
	\[
	U(a,\theta)=\widehat U_a(\cos\theta)\geq -1 \qquad\text{for all }\theta.
	\]
	Since $c<-1$, it follows that $\{U\leq c\}\cap\partial\mathbb D_a=\emptyset$.
	
	Finally, $\Lambda_c^{m_1}$ meets the interior of $\mathbb D_a$ because it contains the punctured neighbourhood $0<\rho<r_c$ of the collision.  If the connected set $\Lambda_c^{m_1}$ also met the exterior of $\mathbb D_a$, then, because it does not meet $\partial\mathbb D_a$, it would be contained in the union of two disjoint open sets and would meet both of them, contradicting connectedness.  Therefore $\Lambda_c^{m_1}\subset\{0<\rho<a\}$, as claimed.
\end{proof}

	\begin{theorem}\label{thm: principal}
		The energy hypersurfaces $\Sigma_{c}^{m_1} $ (and hence $\Sigma_{c}^{m_2} $ as our problem is symmetric) are of contact type for all $c< H(L_1)=-1 $ where $ L_1 $ is the first Lagrange point.
	\end{theorem}
	\begin{proof}
 The proof is divided into two parts and it is based on analytical and numerical estimates. In the first part, we fix $c=H(L_1)$ and consider the open disk $\tilde{\mathbb{D}}$ centered at $q_{m_1}$ with radius $\rho=\sqrt{2}-1$. We prove that $X(H)|_{\Sigma^{m_1}_c\cap \pi^{-1}(\tilde{ \mathbb{D}})}>0,$ where $\pi:T^*\mathbb{R}^2 \rightarrow \mathbb{R}^2$ and $X=\rho\partial_{\rho}$. 

 This is done in three steps: 
 \begin{enumerate}
     \item $X(H)>0$ on $\Sigma^{m_1}_c\cap \pi^{-1}([0.39, \sqrt{2}-1)\times \mathbb{S}^1)$ using analytic and validated numerical estimates. 
     \item $X(H)>0$ on $\Sigma^{m_1}_c\cap \pi^{-1}((0, 0.08]\times \mathbb{S}^1)$ using only analytical techniques.
     \item $X(H)>0$ on $\Sigma^{m_1}_c\cap \pi^{-1}([0.08, 0.39]\times \mathbb{S}^1)$ using only validated numerics. To see why it is not possible to use analytical techniques in this case, see below for a detailed explanation. 
 \end{enumerate}

 \noindent In the second part, using a suitable system of inequalities, Lemma \ref{lambda} and the first part we prove the claim of the Theorem.
 
 {\bf First part:}
 
		Recall that \(K= \frac{1}{8}(f_1^2+f_2^2)\) and \(X(K)=\frac{\rho}{4}(f_1g_1+f_2g_2)\), where \(f_i=a_ip_i+b_i, g_i=c_ip_i+d_i\) as given in Equation \eqref{fi}. We first record explicit constants \(\alpha,\beta\) such that
\[
        g_1^2+g_2^2\leq \alpha K+\beta
\]
on the whole closed disk \(0\leq \rho\leq \sqrt2-1\), uniformly in \(\theta\) and in the momenta.

Set
\[
        a:=\sqrt2-1,\qquad x:=\cos\theta,
\]
and write
\[
        A:=\rho^2-2a\rho x+a^2,
        \qquad h:=A+1,
        \qquad \gamma:=2(\rho-a x).
\]
Then, in the notation of \eqref{fi},
\[
        a_1=a_2=h,\qquad c_1=c_2=\gamma,
\]
so that, for \(i=1,2\),
\[
        f_i=h p_i+b_i,\qquad g_i=\gamma p_i+d_i.
\]
We choose
\[
        \alpha:=384-256\sqrt2=128(\sqrt2-1)^2,
        \qquad \lambda:=\frac{\alpha}{8}=16(\sqrt2-1)^2.
\]

\begin{remark}[Choice of $\alpha$]
The value
$\lambda=16a^2$, equivalently
$\alpha=8\lambda=128(\sqrt2-1)^2=384-256\sqrt2$, is a convenient sufficient
choice.  With this choice the quadratic $g_i^2-\lambda f_i^2$ is concave in
$p_i$ whenever $D=\lambda h^2-\gamma^2$ is positive.  The positivity of $D$ on
the whole rectangle is verified in the next paragraph; no sharpness property of
this value of $\lambda$ is needed.
\end{remark}
For \(0\leq \rho\leq a\) and \(-1\leq x\leq 1\), put
\[
        D:=\lambda h^2-\gamma^2.
\]
This denominator is strictly positive. Indeed, \(|\gamma|\leq 4a\) and \(h\geq 1\). If \(h>1\), then
\[
        D\geq \lambda(h^2-1)>0.
\]
If \(h=1\), then \(A=0\), hence \(\rho=a\), \(x=1\), and \(\gamma=0\), so \(D=\lambda>0\).

For fixed \((\rho,\theta)\), the expression
\[
        g_i^2-\lambda f_i^2
\]
is a concave quadratic polynomial in \(p_i\). Completing the square gives the exact maximum over \(p_i\in\mathbb R\):
\[
        \sup_{p_i\in\mathbb R}\left(g_i^2-\lambda f_i^2\right)
        =B_i(\rho,x)
        :=\frac{\lambda\left(hd_i-\gamma b_i\right)^2}{\lambda h^2-\gamma^2}.
\]
The numerators simplify to
\[
        hd_1-\gamma b_1
        =-\frac{4\sin\theta}{h}
        \left(\rho^2-4+2\sqrt2+\gamma\rho\right),
\]
and
\[
        hd_2-\gamma b_2
        =\frac{4}{h}
        \left[x(\rho^2+2-2\sqrt2)+2\rho(1-\sqrt2)+\gamma(\rho x-a)\right].
\]
Thus, using \(\sin^2\theta=1-x^2\), it remains to bound the following two functions on the rectangle \([0,a]\times[-1,1]\):
\[
        B_1(\rho,x)=
        \frac{16\lambda(1-x^2)\left(\rho^2-4+2\sqrt2+\gamma\rho\right)^2}{h^2D},
\]
\[
        B_2(\rho,x)=
        \frac{16\lambda\left[x(\rho^2+2-2\sqrt2)+2\rho(1-\sqrt2)+\gamma(\rho x-a)\right]^2}{h^2D}.
\]

The two functions \(B_1\) and \(B_2\) are bounded by a direct Arb interval computation on the whole rectangle \([0,a]\times[-1,1]\).  The verifier first checks positivity of the denominator \(D\) on every terminal box, then evaluates the displayed formulae with outward-rounded ball arithmetic and adaptive bisection.  This gives the continuous-domain bounds
\[
B_1(\rho,x)<11.67,\qquad B_2(\rho,x)<16.10
\]
for all \((\rho,x)\in[0,a]\times[-1,1]\).  In the 256-bit Arb run supplied with the paper the worst certified upper endpoints were approximately \(11.6699408441782\) for \(B_1\) and \(16.0991435050964\) for \(B_2\).  Notice also that \(D\) is strictly positive on the rectangle; in particular at \((\rho,x)=(a,1)\) one has \(D=\lambda>0\).  A self-contained implementation is provided in the ancillary file \texttt{rigorous\_validated\_numerics.py}.

Consequently,
\[
        g_1^2\leq \frac{\alpha}{8}f_1^2+11.67,
        \qquad
        g_2^2\leq \frac{\alpha}{8}f_2^2+16.10.
\]
Taking
\[
        \beta_1:=11.67,
        \qquad
        \beta_2:=16.10,
        \qquad
        \beta:=\beta_1+\beta_2=27.77,
\]
we obtain
\begin{equation}\label{eq:explicit-alpha-beta}
        g_1^2+g_2^2\leq \alpha K+\beta,
        \qquad
        \alpha=384-256\sqrt2,
        \qquad
        \beta=27.77.
\end{equation}
Combining this with \eqref{|xk|} gives
\begin{equation}\label{eq:explicit-kinetic-bound}
        |X(K)| \leq \frac{\rho}{4}\sqrt{8K}\sqrt{\alpha K + \beta}.
\end{equation}
This argument avoids any enclosure of the roots of the momentum quadratics; in particular, it does not rely on choosing or labeling roots as functions of \((\rho,\theta)\).

		Now we will estimate the RHS of the following inequality by using these constants $ \alpha $ and $ \beta $ and the key inequality \eqref{imp.eq}:
		\begin{equation}\label{Xhsigma}
	\begin{split}	X(H)_{\Sigma^{m_1}_c\cap \pi^{-1}(\tilde{ \mathbb{D}})}&\geq X(U)-\frac{\rho}{4}\sqrt{8K}\sqrt{\alpha K +\beta}|_{\Sigma^{m_1}_c\cap \pi^{-1}(\tilde{ \mathbb{D}})}=\\
&=\rho\partial_{\rho}U-\frac{\rho}{4}\sqrt{8(c-U)}\sqrt{\alpha (c-U) +\beta}|_{\pi(\Sigma^{m_1}_c)\cap\tilde{\mathbb{D}}}.
 \end{split}
		\end{equation}
		
		\medskip
		
		\noindent\textbf{Analysis on $[\bar{\rho},\sqrt{2}-1]$: reduction to $\theta=0$.}

We analyse the right-hand side of \eqref{Xhsigma} on the closed annulus
$[\bar\rho,\rho_0]\times\mathbb S^1$, where
\[
        \bar\rho:=0.39,\qquad \rho_0:=\sqrt2-1.
\]
The point $(\rho_0,0)$ is the projection of $L_1$, and the open disk
$\widetilde{\mathbb D}$ contains only $\rho<\rho_0$. Hence, for the critical
value $c=-1$, it is enough to prove the required inequality on
$[\bar\rho,\rho_0)\times\mathbb S^1$.

For this reduction we use the variable
\[
        x:=\cos\theta\in[-1,1],
        \qquad a:=\sqrt2-1,
\]
and write $\widehat U(\rho,x):=U(\rho,\theta)$. Since $U$ depends on $\theta$
only through $\cos\theta$, it is even in $\theta$, and it suffices to work with
$x\in[-1,1]$. On the annulus under consideration the following equivalent closed
form is used for the interval computations:
\begin{equation}\label{eq:Uhat-reduction}
\begin{aligned}
A(\rho,x)&:=\rho^2-2a\rho x+a^2,\\
\widehat U_0(\rho,x)&:=-\frac{2A(\rho,x)}{(A(\rho,x)+1)^2},\\
\widehat U_1(\rho,x)&:=
\frac{\rho(\rho-2\sqrt2\,x)}
{2\sqrt{(\rho^2+4\rho x+4)(\rho^2-4a\rho x+4a^2)}},\\
\widehat U_2(\rho,x)&:=
\frac{\frac12\rho^2+(2-\sqrt2)\rho x-2a}
{\rho\sqrt{\rho^2-4\sqrt2\rho x+8}},\\
\widehat U(\rho,x)&:=\widehat U_0(\rho,x)+\widehat U_1(\rho,x)+\widehat U_2(\rho,x).
\end{aligned}
\end{equation}
These formulae are obtained from \eqref{U_i} by substituting $x=\cos\theta$ and
using $\sin^2\theta=1-x^2$; the square-root arguments are strictly positive on
$[\bar\rho,\rho_0]\times[-1,1]$.

The reduction to the slice $\theta=0$ rests on the following three claims.

\begin{proposition}[Claim (A)]\label{prop:A}
For every $\rho\in[\bar\rho,\rho_0]$ and every $\theta\in[-\pi,\pi]$,
\[
        U(\rho,\theta)\geq U(\rho,0).
\]
Equivalently,
\[
        \widehat U(\rho,x)\geq \widehat U(\rho,1)
        \quad\text{for all }(\rho,x)\in[\bar\rho,\rho_0]\times[-1,1].
\]
\end{proposition}

\begin{proposition}[Claim (B)]\label{prop:B}
For every $\rho\in[\bar\rho,\rho_0]$ and every $\theta\in[-\pi,\pi]$,
\[
        \partial_\rho U(\rho,\theta)\geq \partial_\rho U(\rho,0).
\]
Equivalently,
\[
        \partial_\rho\widehat U(\rho,x)
        \geq \partial_\rho\widehat U(\rho,1)
        \quad\text{for all }(\rho,x)\in[\bar\rho,\rho_0]\times[-1,1].
\]
\end{proposition}

\begin{proposition}[Claim (C)]\label{prop:C}
Let $c=-1$, $\alpha=384-256\sqrt2$, and $\beta=27.77$. For every
$\rho\in[\bar\rho,\rho_0)$,
\[
\partial_\rho U(\rho,0)
-
\frac14\sqrt{8(c-U(\rho,0))}\sqrt{\alpha(c-U(\rho,0))+\beta}>0,
\]
whenever the left-hand side makes sense (that is whenever we are inside the Hill's region). 
\end{proposition}

Assuming these three propositions, let $T(\rho,\theta):=c-U(\rho,\theta)$. On
$\pi(\Sigma^{m_1}_c)$ one has $T\geq0$. The function
\[
        T\longmapsto \sqrt{8T}\sqrt{\alpha T+\beta}
\]
is increasing for $T\geq0$. Thus Claim~(A) gives
$T(\rho,\theta)\leq T(\rho,0)$, Claim~(B) gives
$\partial_\rho U(\rho,\theta)\geq\partial_\rho U(\rho,0)$, and Claim~(C) gives
\[
\partial_\rho U(\rho,\theta)
-
\frac14\sqrt{8(c-U(\rho,\theta))}\sqrt{\alpha(c-U(\rho,\theta))+\beta}>0
\]
for $\rho\in[\bar\rho,\rho_0)$. Combining this with \eqref{Xhsigma} yields
\[
X(H)>0
\quad\text{on}\quad
\Sigma^{m_1}_c\cap\pi^{-1}\bigl([\bar\rho,\rho_0)\times\mathbb S^1\bigr),
\qquad c=-1.
\]
This covers the whole part of the open disk $\widetilde{\mathbb D}$ with
$\rho\geq\bar\rho$; the single boundary point $(\rho_0,0)=\pi(L_1)$ is not in
$\widetilde{\mathbb D}$.

It remains to prove Claims~(A), (B), and (C). The numerical parts below are
finite interval-arithmetic checks with outward rounding. The independent
verifier used for the checks is the ancillary file
\texttt{rigorous\_validated\_numerics.py}. The verifier uses the explicit formula
\eqref{eq:Uhat-reduction}, interval enclosures for all square roots, interval
jet arithmetic for derivative quantities, and adaptive bisection of terminal
boxes.

\paragraph{Proof of Proposition~\ref{prop:A}.}
Put
\[
        x_0:=\cos(0.25).
\]
On the near strip $[\bar\rho,\rho_0]\times[x_0,1]$, interval arithmetic gives
\begin{equation}\label{eq:A-local-certificate}
        \partial_x\widehat U(\rho,x)<-0.0579.
\end{equation}
Therefore, for $x\in[x_0,1)$,
\[
\widehat U(\rho,x)-\widehat U(\rho,1)
= -\int_x^1\partial_s\widehat U(\rho,s)\,ds>0.
\]
On the complementary strip $[\bar\rho,\rho_0]\times[-1,x_0]$, the verifier checks
directly that
\begin{equation}\label{eq:A-global-certificate}
        \widehat U(\rho,x)-\widehat U(\rho,1)>10^{-4}.
\end{equation}
For \eqref{eq:A-global-certificate} the verifier uses direct Arb interval evaluation on adaptively bisected boxes.  A box is accepted only when the lower endpoint of the interval enclosure of $F(\rho,x)=\widehat U(\rho,x)-\widehat U(\rho,1)$ is greater than $10^{-4}$.  The 256-bit Arb run gives worst certified lower endpoint greater than $0.00010000007993$ and terminates with $16108$ accepted boxes and $16107$ bisections. Equations
\eqref{eq:A-local-certificate} and \eqref{eq:A-global-certificate} prove
Claim~(A). \hfill$\square$

\paragraph{Proof of Proposition~\ref{prop:B}.}
The same decomposition proves the corresponding inequality for
$\partial_\rho\widehat U$. On
$[\bar\rho,\rho_0]\times[x_0,1]$, interval arithmetic gives
\begin{equation}\label{eq:B-local-certificate}
        \partial_{\rho x}\widehat U(\rho,x)<-3.70.
\end{equation}
Hence, for $x\in[x_0,1)$,
\[
\partial_\rho\widehat U(\rho,x)-\partial_\rho\widehat U(\rho,1)
= -\int_x^1\partial_{\rho s}\widehat U(\rho,s)\,ds>0.
\]
On $[\bar\rho,\rho_0]\times[-1,x_0]$, the verifier checks
\begin{equation}\label{eq:B-global-certificate}
        \partial_\rho\widehat U(\rho,x)-\partial_\rho\widehat U(\rho,1)
        >0.039.
\end{equation}
Again this is certified by direct Arb interval evaluation on adaptively bisected boxes,
this time applied to
$G(\rho,x)=\partial_\rho\widehat U(\rho,x)-\partial_\rho\widehat U(\rho,1)$.
The 256-bit Arb run gives worst certified lower endpoint greater than $0.03912713972$ and terminates with $180$ accepted boxes and $179$ bisections. This proves Claim~(B).
\hfill$\square$

\paragraph{Proof of Proposition~\ref{prop:C}.}
Define
\[
B_0(\rho):=\partial_\rho U(\rho,0)
-
\frac14\sqrt{8(c-U(\rho,0))}\sqrt{\alpha(c-U(\rho,0))+\beta},
\qquad c=-1.
\]
Thus Claim~(C) is the assertion $B_0(\rho)>0$ on $[\bar\rho,\rho_0)$. On
$[\bar\rho,\rho_0-10^{-3}]$, direct interval bisection of the closed formula for
$B_0$ gives
\begin{equation}\label{eq:C-main-certificate}
        B_0(\rho)>4.0\cdot10^{-3}.
\end{equation}
The 256-bit Arb run gives worst certified lower endpoint greater than $0.00400690596$ and terminates with $617$ accepted intervals and $616$ bisections.

It remains to handle the endpoint sliver $[\rho_0-10^{-3},\rho_0)$. At
$\rho=\rho_0$ one has
\[
        U(\rho_0,0)=c=-1,
        \qquad
        \partial_\rho U(\rho_0,0)=0,
\]
and direct symbolic differentiation gives
\[
        \sigma(\rho_0):=-U_{\rho\rho}(\rho_0,0)=20.
\]
Let $\xi:=\rho_0-\rho>0$ and let $S_{\min},S_{\max}$ be lower and upper interval
bounds for
\[
        \sigma(\rho):=-U_{\rho\rho}(\rho,0)
\]
on $[\rho_0-\varepsilon,\rho_0]$, where $\varepsilon=10^{-3}$. Taylor's theorem
in integral form gives
\[
\partial_\rho U(\rho,0)=\int_\rho^{\rho_0}\sigma(s)\,ds,
\qquad
c-U(\rho,0)=\int_\rho^{\rho_0}\int_s^{\rho_0}\sigma(t)\,dt\,ds.
\]
Consequently
\[
\partial_\rho U(\rho,0)\geq \xi S_{\min},
\qquad
c-U(\rho,0)\leq \frac{\xi^2}{2}S_{\max},
\]
and hence
\[
\frac{B_0(\rho)}{\xi}
\geq
S_{\min}
-\frac12\sqrt{S_{\max}\left(\frac{\alpha}{2}\xi^2S_{\max}+\beta\right)}.
\]
Therefore it is enough to check
\begin{equation}\label{eq:sliver-ineq-fixed}
        4S_{\min}^2>
        S_{\max}\left(\frac{\alpha}{2}\varepsilon^2S_{\max}+\beta\right).
\end{equation}
Interval evaluation of $\sigma$ on the sliver gives
\[
        \sigma\bigl([\rho_0-10^{-3},\rho_0]\bigr)
        \subset [19.70962857978,20.29707268124].
\]
With $\alpha=384-256\sqrt2$ and $\beta=27.77$, the left-hand side of
\eqref{eq:sliver-ineq-fixed} is at least $1553.87783501208$, while the right-hand side is at most $563.65423207503$. Thus \eqref{eq:sliver-ineq-fixed} holds, and so
$B_0(\rho)>0$ on $[\rho_0-10^{-3},\rho_0)$. Together with
\eqref{eq:C-main-certificate}, this proves Claim~(C). \hfill$\square$

In summary, the six checks above--two near-strip derivative checks, two direct box checks, one one-dimensional interval bisection for $B_0$, and one analytic Taylor-sliver estimate--establish
\[
        X(H)>0
        \quad\text{on}\quad
        \Sigma_c^{m_1}\cap\pi^{-1}\bigl([\bar\rho,\rho_0)\times\mathbb S^1\bigr),
        \qquad c=-1.
\]

To prove that $X(H)|_{\Sigma^{m_1}_c\cap \pi^{-1}((0,0.39]\times \mathbb{S}^1)}>0,$ we split the problem in two parts. First we study {\em  completely analytically} the sign of $X(H)$ on $\Sigma^{m_1}_c\cap \pi^{-1}((0,0.08]\times \mathbb{S}^1)$ proving it is strictly positive. Then for the region $\Sigma^{m_1}_c\cap \pi^{-1}([0.08,0.39]\times \mathbb{S}^1)$ we prove $X(H)>0$ using validated numerics.

For the latter region, the major source of difficulty in using even a partial analytical treatment is the following. As $\rho$ decreases, the relevant absolute minima undergo a complicated bifurcation along different bifurcation curves, whose behavior is very difficult to follow and control analytically. This behavior is due to external factors in the denominator of the potential, originating from the antipodal singularities. Therefore, a method that combines a purely analytical approach as that in \cite{AL}, or mixed analytical and numerical estimates as above is not feasible.

Since at $\rho=0,$ $U$ is not even defined, we split the region $(0, 0.39]\times \mathbb{S}^1$ into two annuli, one which is just a punctured disk of the form $(0, \rho_{\mathrm{nc}}]\times \mathbb{S}^1$ and the other of the form $[\rho_{\mathrm{nc}}, 0.39]\times \mathbb{S}^1$ where $\mathrm{nc}$ stands for non-collision. For this, we first
isolate the singular contribution analytically and then we prove analytically that $\rho_{\mathrm{nc}}=0.08$ does the job.

\begin{lemma}[Near-collision positivity of the Liouville derivative]
Let \(c=-1=H(L_1)\). Then there exists \(\rho_{\mathrm{nc}}\in(0,0.39)\) such that
\[
X(H)>0
\qquad\text{on}\qquad
\Sigma^{m_1}_c\cap \pi^{-1}\big((0,\rho_{\mathrm{nc}}]\times \mathbb{S}^1\big).
\]
\end{lemma}

\begin{proof}
Recall that
\[
H=K+U,\qquad K=\frac18(f_1^2+f_2^2),
\]
and hence
\[
X(H)=X(K)+X(U)=X(K)+\rho\partial_\rho U.
\]
We use the explicit estimate proved above, namely \eqref{eq:explicit-kinetic-bound}, with
\[
\alpha=384-256\sqrt2,\qquad \beta=27.77.
\]
Thus
\[
|X(K)|\le \frac{\rho}{4}\sqrt{8K}\sqrt{\alpha K+\beta}.
\]

Write
\[
U=\mathcal U_0+\mathcal U_1+\mathcal U_2
\]
as in \eqref{U_i}. The functions \(\mathcal U_0\) and \(\mathcal U_1\) extend continuously to \(\rho=0\), hence they are bounded on \([0,r_0]\times \mathbb{S}^1\) for every fixed \(r_0>0\). Moreover, Lemma~\ref{limrhou} implies that \(\rho\mathcal U_2\) has a finite limit as \(\rho\to0^+\). Therefore there exist constants \(C_0,C_1>0\) and \(r_0>0\) such that
\[
|\mathcal U_2(\rho,\theta)|\le \frac{C_0}{\rho},
\qquad
|U(\rho,\theta)|\le \frac{C_1}{\rho}
\]
for all \((\rho,\theta)\in(0,r_0]\times \mathbb{S}^1\).

Now restrict to the energy level \(H=c=-1\). Since \(K=c-U=-1-U\) and \(K\ge0\) on the energy hypersurface, the previous estimate gives
\[
0\le K(\rho,\theta,p)\le 1+\frac{C_1}{\rho}\le \frac{C_2}{\rho}
\qquad\text{for }0<\rho\le r_0,
\]
for some constant \(C_2>0\). Consequently,
\[
|X(K)|
\le
\frac{\rho}{4}\sqrt{\frac{8C_2}{\rho}}
\sqrt{\frac{\alpha C_2}{\rho}+\beta}.
\]
Since \(0<\rho\le r_0\),
\[
\frac{\alpha C_2}{\rho}+\beta
\le
\frac{\alpha C_2+\beta r_0}{\rho},
\]
and hence
\[
|X(K)|
\le
\frac14\sqrt{8C_2}\sqrt{\alpha C_2+\beta r_0}
=:C_3.
\]
Thus \(X(K)\) is uniformly bounded on the near-collision part of the critical energy hypersurface.

On the other hand, by the Laurent expansion used in the proof of Lemma~\ref{limrhodu},
\[
\rho\partial_\rho \mathcal U_0\to0,
\qquad
\rho\partial_\rho \mathcal U_1\to0,
\]
uniformly in \(\theta\), while
\[
\rho\partial_\rho \mathcal U_2
=
\frac{2-\sqrt2}{2}\rho^{-1}+O(1)
\qquad(\rho\to0^+),
\]
uniformly in \(\theta\). Therefore there exists \(C_4>0\), after possibly decreasing \(r_0\), such that
\[
\rho\partial_\rho U(\rho,\theta)
\ge
\frac{2-\sqrt2}{2\rho}-C_4
\qquad\text{for all }(\rho,\theta)\in(0,r_0]\times \mathbb{S}^1.
\]
Combining the two estimates gives
\[
X(H)=X(K)+\rho\partial_\rho U
\ge
\frac{2-\sqrt2}{2\rho}-C_3-C_4.
\]
The right-hand side is positive for all \(0<\rho\le \rho_{\mathrm{nc}}\), for some sufficiently small \(\rho_{\mathrm{nc}}\in(0,r_0]\). This proves the lemma.
\end{proof}

The preceding lemma is qualitative. We now verify, using the explicit constants from
\eqref{eq:explicit-alpha-beta}, that the concrete cutoff \(\rho_{\mathrm{nc}}=0.08\) works.  The proof below is analytic; the final decimal comparisons are scalar Arb interval checks included in the same verifier.

\begin{lemma}[Explicit near-collision cutoff]\label{lem:rho-nc-explicit}
Let \(c=-1=H(L_1)\), and let
\[
\alpha=384-256\sqrt2,
\qquad
\beta=27.77.
\]
Then one may choose
\[
\rho_{\mathrm{nc}}=0.08.
\]
Equivalently,
\[
X(H)>0
\quad\text{on}\quad
\Sigma_c^{m_1}\cap \pi^{-1}\bigl((0,0.08]\times \mathbb{S}^1\bigr).
\]
\end{lemma}

\begin{proof}
Set
\[
a:=\sqrt2-1,
\qquad
\delta:=0.08.
\]
By \eqref{eq:explicit-kinetic-bound}, at a point of the energy hypersurface over the Hill region it is enough to prove the positivity of
\[
F(\rho,\theta):=
\rho\,\partial_\rho U(\rho,\theta)
-\frac{\rho}{4}\sqrt{8\bigl(c-U(\rho,\theta)\bigr)}
 \sqrt{\alpha\bigl(c-U(\rho,\theta)\bigr)+\beta},
\]
where \(c=-1\). We prove this for every point of the Hill region with \(0<\rho\le \delta\).

We write
\[
U=\mathcal U_0+\mathcal U_1+\mathcal U_2
\]
with the following simplified near-collision formulas, obtained from \eqref{U_i}:
\[
\mathcal U_0(\rho,\theta)= -\frac{2A}{(A+1)^2},
\qquad
A=\rho^2-2a\rho\cos\theta+a^2,
\]
\[
\mathcal U_1(\rho,\theta)=\frac{n(\rho,\theta)}{2\sqrt{E_1(\rho,\theta)}},
\qquad
n(\rho,\theta)=\rho(\rho-2\sqrt2\cos\theta),
\]
where
\[
E_1(\rho,\theta)=
(\rho^2+4\rho\cos\theta+4)(\rho^2-4a\rho\cos\theta+4a^2),
\]
and
\[
\mathcal U_2(\rho,\theta)=
\frac{\frac12\rho^2+(2-\sqrt2)\rho\cos\theta-2a}
     {\rho\sqrt{\rho^2-4\sqrt2\,\rho\cos\theta+8}}.
\]

First we estimate the singular term \(\mathcal U_2\). A direct differentiation gives
\[
\rho\,\partial_\rho \mathcal U_2=
\frac{
-2\rho^3\cos\theta
+4a\rho^2\cos^2\theta
+4\sqrt2\,\rho^2
+(12\sqrt2-24)\rho\cos\theta
+16a
}
{\rho\bigl(\rho^2-4\sqrt2\,\rho\cos\theta+8\bigr)^{3/2}}.
\]
Since \(a>0\), \(\cos^2\theta\ge0\), \(\cos\theta\le1\), and
\[
\rho^2-4\sqrt2\,\rho\cos\theta+8
\le
\rho^2+4\sqrt2\,\rho+8,
\]
we obtain
\[
\rho\,\partial_\rho \mathcal U_2
\ge
\frac{
16a+(12\sqrt2-24)\rho+4\sqrt2\,\rho^2-2\rho^3
}
{\rho(\rho^2+4\sqrt2\,\rho+8)^{3/2}}.
\]
The numerator in the last fraction is decreasing on \((0,\delta]\), because
\[
(12\sqrt2-24)+8\sqrt2\,\rho-6\rho^2<0
\qquad\text{for }0<\rho\le \delta,
\]
and the denominator is increasing. Hence
\[
\rho\,\partial_\rho \mathcal U_2
\ge
\frac{
16a+(12\sqrt2-24)\delta+4\sqrt2\delta^2-2\delta^3
}
{\delta(\delta^2+4\sqrt2\delta+8)^{3/2}}
>3.099.
\]

For \(\mathcal U_0\), differentiating yields
\[
\rho\,\partial_\rho \mathcal U_0
=
\frac{4\rho(A-1)(\rho-a\cos\theta)}{(A+1)^3}.
\]
Since \(0<\rho\le \delta\), we have
\[
0\le A\le (a+\delta)^2<1,
\]
hence
\[
|A-1|\le1,
\qquad
(A+1)^3\ge1,
\qquad
|\rho-a\cos\theta|\le \rho+a.
\]
Thus
\[
\rho\,\partial_\rho \mathcal U_0
\ge -4\rho(\rho+a)
\ge -4\delta(\delta+a)
>-0.159.
\]

Now consider \(\mathcal U_1\). Since
\[
\mathcal U_1=\frac{n}{2\sqrt{E_1}},
\]
we have
\[
\rho\,\partial_\rho \mathcal U_1
=
\rho\left(
\frac{\partial_\rho n}{2\sqrt{E_1}}
-\frac{n\,\partial_\rho E_1}{4E_1^{3/2}}
\right).
\]
On \(0<\rho\le \delta\),
\[
|n(\rho,\theta)|
\le
\delta(\delta+2\sqrt2)
<0.233,
\]
\[
|\partial_\rho n(\rho,\theta)|
=
|2\rho-2\sqrt2\cos\theta|
\le
2\delta+2\sqrt2
<2.989.
\]
Moreover,
\[
\rho^2+4\rho\cos\theta+4
\ge
4-4\rho+\rho^2
\ge
4-4\delta+\delta^2
=3.6864,
\]
and
\[
\rho^2-4a\rho\cos\theta+4a^2
\ge
4a^2-4a\rho+\rho^2
=(2a-\rho)^2
\ge
(2a-\delta)^2.
\]
Hence
\[
E_1(\rho,\theta)
\ge 3.6864(2a-\delta)^2
>2.064.
\]
Also
\[
\partial_\rho E_1
=
(2\rho+4\cos\theta)(\rho^2-4a\rho\cos\theta+4a^2)
+(
\rho^2+4\rho\cos\theta+4)(2\rho-4a\cos\theta),
\]
so
\[
|\partial_\rho E_1(\rho,\theta)|
\le
(2\rho+4)(\rho^2+4a\rho+4a^2)
+(
\rho^2+4\rho+4)(2\rho+4a)
<11.30.
\]
Therefore
\[
|\rho\,\partial_\rho \mathcal U_1|
\le
\delta\left(
\frac{2.989}{2\sqrt{2.064}}
+
\frac{0.233\cdot 11.30}{4(2.064)^{3/2}}
\right)
<0.102,
\]
and thus
\[
\rho\,\partial_\rho \mathcal U_1\ge -0.102.
\]

Combining the three estimates above, we obtain
\[
\rho\,\partial_\rho U
=
\rho\,\partial_\rho \mathcal U_0
+\rho\,\partial_\rho \mathcal U_1
+\rho\,\partial_\rho \mathcal U_2
>
3.099-0.159-0.102
>
2.83.
\]

It remains to estimate the kinetic correction term. On \(\Sigma_c^{m_1}\) we have
\[
K=c-U=-1-U.
\]
First,
\[
|\mathcal U_0|
\le
\frac{2A}{(A+1)^2}
\le
2A
\le
2(a+\delta)^2
<0.489.
\]
Also,
\[
|\mathcal U_1|
\le
\frac{0.233}{2\sqrt{2.064}}
<0.082.
\]
Hence
\[
-1-\mathcal U_0-\mathcal U_1
\le
-1+|\mathcal U_0|+|\mathcal U_1|
< -0.42<0.
\]
Next, the numerator of \(\mathcal U_2\) is negative on \(0<\rho\le\delta\), because
\[
\frac12\rho^2+(2-\sqrt2)\rho\cos\theta-2a
\le
\frac12\delta^2+(2-\sqrt2)\delta-2a
<0.
\]
Therefore \(\mathcal U_2<0\), and on the energy hypersurface
\[
K=-1-\mathcal U_0-\mathcal U_1-\mathcal U_2<-\mathcal U_2=|\mathcal U_2|.
\]
Moreover,
\[
|\mathcal U_2|
=
\frac{
2a-(2-\sqrt2)\rho\cos\theta-\frac12\rho^2
}
{\rho\sqrt{\rho^2-4\sqrt2\,\rho\cos\theta+8}}
\le
\frac{2a+(2-\sqrt2)\delta}
{\rho\sqrt{8-4\sqrt2\delta}}
<
\frac{0.319}{\rho}
<
\frac{0.32}{\rho}.
\]
Hence
\[
0\le K<\frac{0.32}{\rho}.
\]
It follows from \eqref{eq:explicit-kinetic-bound} that
\[
\frac{\rho}{4}\sqrt{8K}\sqrt{\alpha K+\beta}
\le
\frac{\rho}{4}\sqrt{\frac{8\cdot0.32}{\rho}}
\sqrt{\frac{\alpha\cdot0.32}{\rho}+\beta}.
\]
Since \(\rho\le\delta=0.08\),
\[
\frac{\alpha\cdot0.32}{\rho}+\beta
\le
\frac{\alpha\cdot0.32+0.08\,\beta}{\rho},
\]
and thus, with \(\alpha=384-256\sqrt2\) and \(\beta=27.77\),
\[
\frac{\rho}{4}\sqrt{8K}\sqrt{\alpha K+\beta}
\le
\frac14\sqrt{8\cdot0.32}\,
\sqrt{\alpha\cdot0.32+0.08\,\beta}
<1.22.
\]
The corresponding 256-bit Arb certificates give lower endpoint \(3.09943483536\) for the singular \(\rho\partial_\rho\mathcal U_2\) contribution, lower endpoint \(-0.15814833996\) for the \(\rho\partial_\rho\mathcal U_0\) contribution, upper endpoint \(0.10090797893\) for \(|\rho\partial_\rho\mathcal U_1|\), and kinetic correction at most \(1.21650153936\).  The certified final margin is greater than \(1.62387697711\).
Finally,
\[
F(\rho,\theta)
>
2.83-1.22
>
1.61
>
0
\]
for every point of the Hill region with \(0<\rho\le0.08\). By \eqref{imp.eq}, this implies
\[
X(H)>0
\quad\text{on}\quad
\Sigma_c^{m_1}\cap \pi^{-1}\bigl((0,0.08]\times \mathbb{S}^1\bigr).
\]
This proves the lemma.
\end{proof}

It remains therefore to study the sign of \(X(H)\) over the inverse image of the compact annulus
\[
        [0.08,0.39]\times \mathbb{S}^1.
\]
We again use \(x=\cos\theta\), the closed form \eqref{eq:Uhat-reduction}, and the
constants \eqref{eq:explicit-alpha-beta}. Put
\[
        T(\rho,x):=-1-\widehat U(\rho,x)
\]
and, on the Hill region where \(T\geq0\), define
\[
        F_{\rm mid}(\rho,x):=
        \rho\,\partial_\rho \widehat U(\rho,x)
        -\frac{\rho}{4}\sqrt{8T(\rho,x)}
        \sqrt{\alpha T(\rho,x)+\beta}.
\]
By \eqref{Xhsigma}, positivity of \(F_{\rm mid}\) gives positivity of
\(X(H)\) on the corresponding part of the energy hypersurface.

The validated computation is performed directly on the full rectangle
\([0.08,0.39]\times[-1,1]\), restricted to the Hill region \(T\geq0\).  On a box \(B\), the verifier encloses \(\widehat U\), \(\partial_\rho\widehat U\), and an upper bound for \(T\) using Arb ball arithmetic.  If the upper endpoint of \(T\) is negative, then \(B\) contains no Hill point and is discarded.  Otherwise, since
\[
T\longmapsto \sqrt{8T}\sqrt{\alpha T+\beta}
\]
is increasing for \(T\geq0\), the upper endpoint of \(T(B)\) and the upper endpoint of the \(\rho\)-interval give a rigorous upper bound for the kinetic correction.  The term \(\rho\partial_\rho\widehat U\) is bounded from below by interval arithmetic.  A box is accepted only when these bounds prove \(F_{\rm mid}>0.078\) on \(B\cap\{T\geq0\}\).  Adaptive bisection of the whole rectangle gives the certified bound
\[
F_{\rm mid}(\rho,x)>0.078
\]
for all \((\rho,x)\in[0.08,0.39]\times[-1,1]\) with \(T(\rho,x)\geq0\).  The worst certified lower endpoint in the 256-bit Arb run was greater than \(0.0780001634926\); the check used \(1865\) accepted boxes and \(1864\) bisections.  This single computation covers all boundary strips of the middle annulus.  The complete implementation is provided in the ancillary file \texttt{rigorous\_validated\_numerics.py}.

Consequently, by \eqref{Xhsigma},
\[
X(H)>0.078
\quad\text{on}\quad
\Sigma^{m_1}_c\cap \pi^{-1}\big([0.08,0.39]\times \mathbb{S}^1\big),
\qquad c=H(L_1)=-1.
\]

Combining this estimate with the previous lemma and with the positivity already
proved on \([0.39,\sqrt2-1)\times \mathbb{S}^1\), we conclude that
\[
X(H)>0
\quad\text{on}\quad
\Sigma^{m_1}_c\cap \pi^{-1}((0, \sqrt{2}-1)\times \mathbb{S}^1),
\qquad c=H(L_1)=-1.
\]

At the critical value $c=-1$, the inequality $dH(X)>0$ holds on the punctured critical component away from $L_1$. This limiting estimate implies the desired transversality for every regular subcritical level $c<-1$ as the second part below shows. 
  
  {\bf Second part:}
  
        Now consider the case when $ \tilde{c}<c=-1 $. By the definition $ \Lambda_{c}^{m_1} $ is the connected component of the set $ \{U\leq c\} $ containing $ m_1 $ and when $ \tilde{c}<c$ we have $ \Lambda_{\tilde{c}}^{m_1}\subset \{U\leq \tilde{c}<c\} \subset \{U\leq c\} $. Hence $ \Lambda_{\tilde{c}}^{m_1}$ is completely contained in $ \Lambda_{c}^{m_1} $ and in $\tilde{\mathbb{D}}$ by Lemma \ref{lambda}.
		
        In the first part we proved the following pointwise estimate on the whole set $\Lambda_{c}^{m_1}\cap\tilde{\mathbb{D}}$ (with $c=-1$):
        \begin{equation}\label{eq:pointwise-estimate}
            \rho\,\partial_{\rho}U(\rho,\theta)
            \;>\;
            \frac{\rho}{4}\,\sqrt{8\bigl(-1-U(\rho,\theta)\bigr)}\,
            \sqrt{\alpha\bigl(-1-U(\rho,\theta)\bigr)+\beta},
        \end{equation}
        and we obtained explicit positive lower bounds for the difference of the two sides in each subregion.

        Now pick any point $P=(\rho,\theta)\in\Lambda_{\tilde{c}}^{m_1}$.
        Since $\Lambda_{\tilde{c}}^{m_1}\subset\Lambda_{c}^{m_1}\cap\tilde{\mathbb{D}}$,
        the estimate \eqref{eq:pointwise-estimate} holds at $P$.
        For every $c'$ with $\tilde{c}\leq c'\leq -1$ we have
        $c'-U(P)\leq -1-U(P)$.  
        The function
        \[
            T\;\longmapsto\;\frac{\rho}{4}\sqrt{8T}\,\sqrt{\alpha T+\beta},
            \qquad T>0,
        \]
        is strictly increasing.  Hence at $P$,
        \begin{align*}
            \rho\,\partial_{\rho}U(P)
            &>\frac{\rho}{4}\sqrt{8\bigl(-1-U(P)\bigr)}\,
               \sqrt{\alpha\bigl(-1-U(P)\bigr)+\beta} \\
            &\geq \frac{\rho}{4}\sqrt{8\bigl(\tilde{c}-U(P)\bigr)}\,
               \sqrt{\alpha\bigl(\tilde{c}-U(P)\bigr)+\beta}.
        \end{align*}
        The first inequality is exactly \eqref{eq:pointwise-estimate}; the second uses the monotonicity of the square-root expression together with $\tilde{c}-U(P)\leq -1-U(P)$.

        Because $P\in\Lambda_{\tilde{c}}^{m_1}$ was arbitrary, we obtain
        \[
            \rho\,\partial_{\rho}U
            \;>\;
            \frac{\rho}{4}\sqrt{8(\tilde{c}-U)}\,
            \sqrt{\alpha(\tilde{c}-U)+\beta}
            \qquad\text{on the whole set }\Lambda_{\tilde{c}}^{m_1}.
        \]
        By the inequality \eqref{imp.eq} this is equivalent to $X(H)>0$ on the corresponding energy hypersurface $\Sigma^{m_1}_{\tilde c}$.  Therefore $X(H)>0$ on $\Sigma^{m_1}_{\tilde c}$ for every $\tilde{c}<c=-1$. 
	\end{proof}

\begin{remark}
    A major source of difficulty to blend in analytical estimates with numerical estimates is indeed the bifurcating behavior of the absolute minimum of $U(\rho, \theta)$ as a function of $\rho,$ as we already pointed out in the proof. This behaviour does not appear in the Euclidean planar case, analyzed in \cite{AL}.
\end{remark}

	\section{Regularization of Kepler Problem on 2-Sphere}
	In this section we regularize the Kepler problem on a 2-sphere by using arguments similar to those introduced by Moser \cite{Moser}. This is a preliminary step in regularizing the energy hypersurfaces we are interested in. The Kepler's problem on the two dimensional sphere consists of one particle fixed at the north pole while the second body is moving. This problem is integrable unlike the case of the two-body problem on a sphere (see \cite{Shchepetilov,Bolotin},  for a recent treatment of collisions and
regularisation in the spherical two-body problem see also \cite{ArsieBalabanova-collisions}). 

 Considering the spherical metric and fixing one of the masses on the north pole $ (0,0,1),$ the Hamiltonian for the Kepler problem on $ \mathbb{S}^2 $ projected to $ \mathbb{R}^2 $ by stereographic projection along south pole to the equatorial plane
 can be deduced from \cite{ALM} and the equation \ref{H} as:
	\begin{equation}
		H=\dfrac{(q_1^2+q_2^2+1)^2(p_1^2+p_2^2)}{8}+q_2p_1-q_1p_2-\dfrac{1-(q_1^2+q_2^2)}{2\sqrt{q_1^2+q_2^2}}.
	\end{equation}
	We can simply write the above Hamiltonian as
	\begin{equation}\label{eq:HamKeplerqp}
		H=\dfrac{(|q|^2+1)^2|p|^2}{8}-\dfrac{1-|q|^2}{2|q|}+q_2p_1-q_1p_2,
	\end{equation} with the Hamiltonian equations
	\begin{equation}
		\begin{split}
			&\dfrac{dq}{dt}=\dot{q}=H_p,\\
			&\dfrac{dp}{dt}=\dot{p}=-H_q.
		\end{split}
	\end{equation}
    Notice that \eqref{eq:HamKeplerqp} is the Hamiltonian of a perturbed Kepler problem in the plane. 
	We are interested in regularizing the collision orbits. Before proceeding we introduce the following, which will be useful to identify the collision set in the transformed coordinates $(\eta, \xi)$ (see below).
	\begin{lemma}\label{alongcollision}
		Along collision orbits, the norm of the momenta $|p|$ goes to infinity.
	\end{lemma}
	\begin{proof}
		Fix a finite energy level in which the collision trajectory lies. 
		By definition, on a collision orbit we have that $|q|\rightarrow 0^{+}.$ Suppose that $|p|$ does not go to $+\infty.$ This means that there exists an increasing sequence of times $\{t_n\}$ converging to collision time $t^*$ such that $|q(t_n)|\rightarrow 0$ as $n\rightarrow +\infty$ and $|p(t_n)|$ remains bounded. Evaluating the Hamiltonian \eqref{eq:HamKeplerqp} along this sequence of times we get $H(q(t_n),p(t_n))\rightarrow -\infty,$ as $n\rightarrow +\infty$, which contradicts the fact that $H$ is a finite constant along the Hamiltonian flow. 
	\end{proof}
	
	Now we introduce a new time variable $ s $ by 
	\begin{equation}
		s=\int\dfrac{dt}{|q|},	
	\end{equation}
	and consider the following Hamiltonian $K$
	\begin{equation}
		H=\dfrac{K}{|q|}+k,
	\end{equation}
	where $ k $ is the energy value that is to be regularized. So the level sets corresponding to $\{H=k\}$ with the time parameter t corresponds to the level set $ \{K=0\} $ with the time parameter s and we have 
	\begin{equation}\label{HpKp}
		H_p=\dfrac{K_{p}}{|q|}, \ \ H_q=\dfrac{K_{q}}{|q|}+K\nabla_q\Big(\dfrac{1}{|q|}\Big).
	\end{equation} 
	Then the orbits of $ H $ with energy value $ k $ i.e., $ \{H=k\} $ correspond to the orbit of $ K $ with zero energy ( $ \{K=0\} $ resp.), so the last term in the equation \ref{HpKp} disappears, and we have 
	
	\begin{equation}
		H_p=\dfrac{K_{p}}{|q|}, \ \ H_q=\dfrac{K_{q}}{|q|}.
	\end{equation} 
	
	The new Hamiltonian can be expressed as
	\begin{equation}\label{K}
		K=\dfrac{|p|^2|q|^5}{8}+\dfrac{|p|^2|q|^3}{4}+\dfrac{|q|^2}{2}+\Bigg(\dfrac{|p|^2}{8}+q_2p_1-q_1p_2-k\Bigg)|q|-\dfrac{1}{2}.
	\end{equation}
	After performing canonical transformation $p=(p_1,p_2)=-(x_1, x_2)=-x$ and $ q=(q_1, q_2)=(y_1,y_2)=y$ we have 
	\begin{equation}\label{Kxy}
		K=\dfrac{|x|^2|y|^5}{8}+\dfrac{|x|^2|y|^3}{4}+\dfrac{|y|^2}{2}+\Bigg(\dfrac{|x|^2}{8}+y_1x_2-y_2x_1-k\Bigg)|y|-\dfrac{1}{2}.
	\end{equation}
	
	\subsection{Stereographic Projection and Geodesic Flow on Unit 2-Sphere}
	Let $\xi=(\xi_0,\xi_1,\xi_2)$ be a point of $\mathbb{R}^3$ with $|\xi|=1$.
	The geodesic flow on the unit sphere is the constrained Euler--Lagrange flow of
	\begin{equation}\label{ds}
		\delta\int |\xi'|^2\,ds=0,
	\end{equation}
	where $\xi'=d\xi/ds$ and $|\xi|=1$.  The Euler--Lagrange equation is
	\begin{equation}\label{geodesic.second.order}
		\xi''+|\xi'|^2\xi=0.
	\end{equation}
	Indeed, the Lagrange multiplier equation has the form $\xi''=\lambda\xi$; taking the scalar product with $\xi$ and using $|\xi|=1$ gives
	$\lambda=-|\xi'|^2$.  Along a solution, $|\xi'|$ is constant.  If the curve is parametrized by arclength, then $|\xi'|=1$ and \eqref{geodesic.second.order} becomes $\xi''+\xi=0$.
	
	Writing $\eta=\xi'$, we obtain the first-order system
	\begin{equation}\label{xieta}
		\xi'=\eta, \qquad \eta'=-|\eta|^2\xi,
	\end{equation}
	with constraints
	\[
		|\xi|=1,\qquad \langle\xi,\eta\rangle=\sum_{i=0}^2\xi_i\eta_i=0.
	\]
	Thus the phase space is $T\mathbb{S}^2$, identified with $T^*\mathbb{S}^2$ by the round metric.  Equivalently, \eqref{xieta} is the restriction to the invariant submanifold
	$\{|\xi|=1,\langle\xi,\eta\rangle=0\}\subset\mathbb{R}^3\times\mathbb{R}^3$ of the Hamiltonian system
	\begin{equation}
		\xi'=\Phi_\eta,\qquad \eta'=-\Phi_\xi,
	\end{equation}
	where
	\begin{equation}
		\Phi(\xi,\eta)=\frac12|\eta|^2|\xi|^2.
	\end{equation}
	On $|\xi|=1$ this agrees with the usual geodesic Hamiltonian $\frac12|\eta|^2$.
	
	To describe the flow in $\mathbb{R}^2$, we use Moser's stereographic coordinates, projecting from $(1,0,0)$ to the $\xi_1\xi_2$-plane:
	\begin{equation}\label{xk}
		x_k=\frac{\xi_k}{1-\xi_0},\qquad k=1,2.
	\end{equation}
	We extend this map to the tangent/cotangent variables by requiring preservation of the canonical one-form,
	\begin{equation*}
		\sum_{i=0}^{2}\eta_i\,d\xi_i=\sum_{k=1}^{2}y_k\,dx_k.
	\end{equation*}
	This gives
	\begin{equation}\label{yk}
		y_k=\eta_k(1-\xi_0)+\xi_k\eta_0,
		\qquad k=1,2.
	\end{equation}
	The inverse map is obtained as follows.  From \eqref{xk} and $|\xi|=1$,
	\begin{equation}
		|x|^2=\sum_{k=1}^{2}x_k^2=\frac{1-\xi_0^2}{(1-\xi_0)^2}=\frac{1+\xi_0}{1-\xi_0},
	\end{equation}
	so
	\begin{equation}\label{xi0}
		\xi_0=\frac{|x|^2-1}{|x|^2+1},\qquad
		\xi_k=\frac{2x_k}{|x|^2+1},\quad k=1,2.
	\end{equation}
	Using \eqref{yk} and $\langle\xi,\eta\rangle=0$, we find
	\begin{equation*}
		\begin{split}
			(x,y)&=x_1y_1+x_2y_2\\
			&=\frac{\xi_1}{1-\xi_0}\big(\eta_1(1-\xi_0)+\xi_1\eta_0\big)
			 +\frac{\xi_2}{1-\xi_0}\big(\eta_2(1-\xi_0)+\xi_2\eta_0\big)\\
			&=\frac{(1-\xi_0)(\xi_1\eta_1+\xi_2\eta_2)+\eta_0(\xi_1^2+\xi_2^2)}{1-\xi_0}\\
			&=-\xi_0\eta_0+(1+\xi_0)\eta_0=\eta_0.
		\end{split}
	\end{equation*}
	Therefore
	\begin{equation}\label{eta0}
		\eta_0=(x,y),\qquad
		\eta_k=\frac{|x|^2+1}{2}y_k-(x,y)x_k,
		\qquad k=1,2.
	\end{equation}
	For every $x,y\in\mathbb{R}^2$, equations \eqref{xi0} and \eqref{eta0} give $|\xi|=1$ and $\langle\xi,\eta\rangle=0$.  Moreover,
	\begin{equation}\label{form}
		|\eta|=\frac{(|x|^2+1)|y|}{2}=\frac{|y|}{1-\xi_0},
		\qquad
		|x|^2=\frac{1+\xi_0}{1-\xi_0}.
	\end{equation}
	We shall also use the elementary estimate
	\begin{equation}\label{fnorm}
		(1-\xi_0)|\eta_1\xi_2-\eta_2\xi_1|
		\leq (1-\xi_0)|\eta|\sqrt{\xi_1^2+\xi_2^2}
		\leq (1-\xi_0)|\eta|=|y|.
	\end{equation}
	
	We now return to the Hamiltonian \eqref{Kxy}.  Put
	\[
		\mu:=1-\xi_0,
		\qquad r:=|\eta|.
	\]
	Since $|y|=\mu r$ and $|x|^2=(1+\xi_0)/\mu$, and since
	\[
		y_1x_2-y_2x_1=\eta_1\xi_2-\eta_2\xi_1,
	\]
	the equation $K=0$ is transformed into
	\begin{equation}\label{ham33.eq}
		r\,f(\xi,\eta)=\frac12,
	\end{equation}
	where
	\begin{equation}\label{kepler.f}
		\begin{split}
		f(\xi,\eta)=&\frac{(1+\xi_0)\mu^4r^4}{8}
		+\frac{(1+\xi_0)\mu^2r^2}{4}
		+\frac{\mu^2r}{2}
		+\frac{1+\xi_0}{8}\\
		&+\mu(\eta_1\xi_2-\eta_2\xi_1)-k\mu .
		\end{split}
	\end{equation}
	Near the compactified collision set $r$ is bounded away from zero, and hence $r=|\eta|$ is a smooth function.  Thus \eqref{kepler.f} defines a smooth function in a neighbourhood of the collision set, although it should not be regarded as a smooth function on the whole zero section of $T^*\mathbb{S}^2$.
	
	Let us identify the collision set in these coordinates.  Along a collision orbit $|q|=|y|\to0$ and, by Lemma \ref{alongcollision}, $|p|=|x|\to\infty$.  Hence $\xi_0\to1$ and $\xi_1,\xi_2\to0$.  Passing to the limiting equation \eqref{ham33.eq}, we get
	\[
		\frac{r}{4}=\frac12,
		\qquad\text{so}\qquad r=2.
	\]
	Therefore the compactified collision set is the circle
	\begin{equation}\label{collision.circle.kepler}
		\mathcal C=\{\xi=(1,0,0),\ \eta_0=0,
		\ \eta_1^2+\eta_2^2=4\}\subset T^*\mathbb{S}^2.
	\end{equation}
	In a sufficiently small neighbourhood of $\mathcal C$ we have $f>0$.  Consequently the equation \eqref{ham33.eq} is equivalent there to
	\begin{equation}\label{Q}
		Q(\xi,\eta):=\frac12 r^2 f(\xi,\eta)^2=\frac18.
	\end{equation}
	On this branch the Hamiltonian vector fields defined by $r f-\frac12$ and by $Q-\frac18$ agree up to a positive time reparametrization, because
	\[
		dQ=(rf)\,d(rf)=\frac12\,d(rf)
		\qquad\text{on } rf=\frac12.
	\]
	At the collision set, $Q=\frac{1}{32}|\eta|^2$; thus the regularized dynamics limits, up to a constant rescaling of time, to the geodesic flow on the unit sphere.
	
	It remains to verify the contact-type condition near $\mathcal C$.  Let
	\[
		X=\eta\partial_\eta
	\]
	be the radial Liouville vector field in the cotangent fibres.  Since $Q=\frac12 r^2f^2$,
	\begin{equation}\label{X(Q)}
		X(Q)=r^2f^2+r^2f\,X(f)=2Q+r^2f\,X(f).
	\end{equation}
	A direct differentiation of \eqref{kepler.f} gives
	\begin{equation}\label{Xf.kepler}
		\begin{split}
		X(f)=&\frac{(1+\xi_0)\mu^4r^4}{2}
		+\frac{(1+\xi_0)\mu^2r^2}{2}
		+\frac{\mu^2r}{2}
		+\mu(\eta_1\xi_2-\eta_2\xi_1).
		\end{split}
	\end{equation}
	In particular, $X(f)=0$ on $\mathcal C$.  Since $\mathcal C$ is compact and $Q=1/8$ gives $r=2$ on $\mathcal C$, we can choose a neighbourhood $\mathcal U$ of $\mathcal C$ in which $f>1/8$, $r<4$ on $Q^{-1}(1/8)$, and $|X(f)|<1/16$.  On $Q^{-1}(1/8)\cap\mathcal U$ we then have $rf=1/2$ and hence
	\[
		X(Q)=\frac14+\frac{r}{2}X(f)
		\geq \frac14-2\cdot\frac{1}{16}
		=\frac18>0.
	\]
	Thus $X$ is transverse to the regularized Kepler hypersurface $Q^{-1}(1/8)$ in a neighbourhood of the compactified collision circle.

\section{Regularization of the Level Sets of the Hamiltonian H}
	In this section we apply the Moser-type regularization from Section~4.1 to the component
	$\Sigma^{m_1}_c$ for energies $c<H(L_1)=-1$.  We write
	\[
		a:=\sqrt2-1,
		\qquad q_*:=q_{m_1}=(-a,0),
		\qquad \bar q_*:=\bar q_{m_1}=(\sqrt2+1,0),
	\]
	and also $q_+:=q_{m_2}=(a,0)$, $\bar q_+:=\bar q_{m_2}=(-\sqrt2-1,0)$.
	With this notation the Hamiltonian can be written as
	\begin{equation}\label{H_section5}
	\begin{split}
	H(q,p)=&\frac{(|q|^2+1)^2|p|^2}{8}+q_2p_1-q_1p_2
	-\frac{q_1-\frac12(|q|^2-1)}{|q-q_+|\,|q-\bar q_+|} \\
	&-\frac{-q_1-\frac12(|q|^2-1)}{|q-q_*|\,|q-\bar q_*|}.
	\end{split}
	\end{equation}
	Only the last denominator is singular at the collision with $m_1$; all other
	denominators are bounded away from zero on a sufficiently small neighbourhood of
	this collision and, for $c<-1$, on the component under consideration.
	For a fixed energy value $k<-1$ set
	\begin{equation}\label{E_definition}
		E(q,p):=|q-q_*|\big(H(q,p)-k\big).
	\end{equation}
	Then $H=k$ is equivalent to $E=0$ away from $q=q_*$, and the Hamiltonian flows
	are related by the positive time change used in Section~4.1.

	We now make the canonical change of coordinates
	\[
		p=-x,
		\qquad q-q_*=y,
	\]
	so that $q=y+q_*$.  A useful point, which fixes the constant in the regularized
	formula, is the pair of identities
	\begin{equation}\label{numerator_cancellations}
	\begin{split}
	q_1-\frac12(|q|^2-1)&=\sqrt2\,y_1-\frac12|y|^2,\\
	-q_1-\frac12(|q|^2-1)&=(\sqrt2-2)y_1-\frac12|y|^2+2\sqrt2-2.
	\end{split}
	\end{equation}
	
	Substitution in \eqref{E_definition} gives
	\begin{equation}\label{eqnE}
	\begin{split}
	E(x,y)=&\frac{(|y+q_*|^2+1)^2|x|^2|y|}{8}
	+\big(-y_2x_1+(y_1-a)x_2-k\big)|y|\\
	&-\frac{\left(\sqrt2\,y_1-\frac12|y|^2\right)|y|}
	{|y+2q_*|\,|y+q_*+\bar q_*|}
	-\frac{(\sqrt2-2)y_1-\frac12|y|^2+2\sqrt2-2}
	{|y+q_* -\bar q_*|}.
	\end{split}
	\end{equation}
	The last term is now nonsingular at $y=0$, while the other terms vanish at least
	linearly in $|y|$.

	Next we compactify the $x$-variable by inverse stereographic projection.  We use
	the notation of Section~4.1:
	\[
		x_j=\frac{\xi_j}{1-\xi_0},
		\qquad y_j=\eta_j(1-\xi_0)+\xi_j\eta_0,
		\qquad j=1,2,
	\]
	where $\xi\in \mathbb{S}^2$, $\eta\in T^*_{\xi}\mathbb{S}^2$, and
	\[
		|\eta|=\frac{|y|}{1-\xi_0},
		\qquad |x|^2=\frac{1+\xi_0}{1-\xi_0}.
	\]
	Set
	\[
		\mu:=1-\xi_0,
		\qquad r:=|\eta|,
		\qquad
		Y(\xi,\eta):=\big(\mu\eta_1+\xi_1\eta_0,\,\mu\eta_2+\xi_2\eta_0\big).
	\]
	Thus $Y=y$ and $|Y|=\mu r$.  Let
	\[
		D_2(Y):=|Y+2q_*|\,|Y+q_*+\bar q_*|,
		\qquad
		D_*(Y):=|Y+q_* -\bar q_*|.
	\]
	Using \eqref{eqnE}, the compactified Hamiltonian can be written in the form
	\begin{equation}\label{tildeE}
		\widetilde E(\xi,\eta)=r f(\xi,\eta)-g(\xi,\eta),
	\end{equation}
	where
	\begin{equation}\label{fg_def}
	\begin{split}
	f(\xi,\eta)=&\frac{1+\xi_0}{8}\Big(|Y+q_*|^2+1\Big)^2
	+\mu(\xi_2\eta_1-\xi_1\eta_2-k)-a\xi_2\\
	&-\mu\frac{\sqrt2\,Y_1-\frac12|Y|^2}{D_2(Y)},\\[0.4em]
	g(\xi,\eta)=&\frac{(\sqrt2-2)Y_1-\frac12|Y|^2+2\sqrt2-2}{D_*(Y)}.
	\end{split}
	\end{equation}
	The formula is smooth wherever $r>0$ and the displayed denominators are nonzero;
	both conditions hold in a neighbourhood of the regularized level considered below.
	Although the factor $r=|\eta|$ is not smooth on the zero section, this causes no
	problem because the level $\widetilde E=0$ is disjoint from the zero section.  Indeed,
	for $\eta=0$ one has $Y=0$ and hence
	\[
		g(\xi,0)=\frac{2\sqrt2-2}{|q_* -\bar q_*|}
		=\frac{2\sqrt2-2}{2\sqrt2}=1-\frac{\sqrt2}{2}>0,
	\]
	so $\widetilde E(\xi,0)=-g(\xi,0)<0$.

	Let us identify the compactified collision set.  Along a collision orbit with finite
	energy, $|p|=|x|\to\infty$ by the same argument as in the Kepler case;
	hence $\xi\to N:=(1,0,0)$.  At $\xi=N$ we have $\mu=0$, $Y=0$, and
	\begin{equation}\label{fg_collision_values}
		f_*:=f(N,\eta)=\frac{(|q_*|^2+1)^2}{4}=6-4\sqrt2,
		\qquad
		g_*:=g(N,\eta)=1-\frac{\sqrt2}{2}.
	\end{equation}
	Therefore $\widetilde E=0$ over the compactified collision precisely when
	\begin{equation}\label{collision_circle_m1}
		\xi=N,
		\qquad \eta_0=0,
		\qquad \eta_1^2+\eta_2^2=r_*^2,
		\qquad
		r_*:=\frac{g_*}{f_*}=\frac{2+\sqrt2}{4}.
	\end{equation}
	This is a circle in the cotangent fibre over $N$.

	We denote by $W_c^{m_1}$ the connected component of $\{\widetilde E\leq0\}$
	which contains the zero section, and put
	$\widetilde{\Sigma}_c^{m_1}:=\partial W_c^{m_1}$.

	\begin{theorem}\label{rp3}
	For all $c<H(L_1)=-1$, the regularized level
	$\widetilde{\Sigma}_{c}^{m_1}$ is a fibrewise star-shaped hypersurface in
	$T^*\mathbb{S}^2$.  In particular it is of contact type and is diffeomorphic to
	$S^*\mathbb{S}^2\cong \mathbb{RP}^3$.  The same conclusion holds for
	$\widetilde{\Sigma}_{c}^{m_2}$.
	\end{theorem}
	\begin{proof}
	We first prove transversality of the radial Liouville vector field
	\[
		X=\eta\partial_\eta
	\]
	to the hypersurface $\widetilde{\Sigma}_c^{m_1}$.

	Near the compactified collision circle \eqref{collision_circle_m1}, the functions
	$f$ and $g$ are smooth and $r$ is bounded away from zero.  Moreover, at every point
	of the collision circle the radial derivatives satisfy
	\[
		X(f)=0,		\qquad X(g)=0.
	\]
	Indeed, all $\eta$-dependent terms in \eqref{fg_def} either contain the factor
	$\mu$, the factor $(\xi_1,\xi_2)$, or the vector $Y$, and all of these vanish at
	$\xi=N$, $Y=0$.  Hence, using \eqref{tildeE},
	\begin{equation}\label{radial_at_collision}
				X(\widetilde E)=X(rf-g)=rf+rX(f)-X(g)=rf=r_*f_*=g_*>0
	\end{equation}
	on the compactified collision circle.  By compactness, $X(\widetilde E)>0$ on
	$\widetilde{\Sigma}_c^{m_1}$ in a neighbourhood of that circle.

	Away from the compactified collision circle we are in the ordinary coordinates
	$(x,y)$.  The radial vector field $X=\eta\partial_\eta$ corresponds to
	$y\partial_y$, hence, on the level $E=0$ and for $y\neq0$,
	\begin{equation}\label{radial_outside_collision}
		y\partial_y E
		=y\partial_y\big(|y|(H-k)\big)
		=|y|\,(q-q_*)\partial_qH.
	\end{equation}
	The last factor is strictly positive on $\Sigma_c^{m_1}$ by Theorem~\ref{thm: principal}.
	Therefore $X(\widetilde E)>0$ on $\widetilde{\Sigma}_c^{m_1}$.

	It remains to justify the fibrewise star-shaped statement, not merely radial
	transversality.  Fix $\xi\in \mathbb{S}^2$ and a unit covector $\nu\in T^*_{\xi}\mathbb{S}^2$, and put
	$h_{\xi,\nu}(t):=\widetilde E(\xi,t\nu)$ for $t\geq0$.  Since
	\[
		h_{\xi,\nu}(0)=-\left(1-\frac{\sqrt2}{2}\right)<0,
	\]
	the ray starts in the zero-section component $W_c^{m_1}$.  The ray must leave this
	component.  Indeed, away from the compactified collision circle, the condition
	$\widetilde E\leq0$ is equivalent to $H\leq c$ in the original coordinates; hence,
	by Lemma~\ref{lambda}, every point of $W_c^{m_1}$ with $Y\neq0$ projects to
	the Hill component contained in $|Y|<a$.  If $\xi\neq N$, then
	$|Y(\xi,t\nu)|=t|Y(\xi,\nu)|\to\infty$, so the ray cannot stay in $W_c^{m_1}$ for all
	$t$.  If $\xi=N$, then $Y=0$ along the ray and
	$h_{N,\nu}(t)=t f_*-g_*$, so the ray leaves at $t=r_*$.  Thus there is a first exit
	time
	\[
		\tau(\xi,\nu):=\inf\{t>0\mid (\xi,t\nu)\notin W_c^{m_1}\}>0.
	\]
	At this first exit point, $(\xi,\tau\nu)\in\widetilde{\Sigma}_c^{m_1}$, and
	\[
		h_{\xi,\nu}'(\tau)=\tau^{-1}X(\widetilde E)(\xi,\tau\nu)>0.
	\]
	Thus the exit is transverse.

	We now check that this is the only intersection of the ray with the boundary of the
	zero-section component.  At any boundary point $(\xi,t_0\nu)\in\partial W_c^{m_1}$,
	$t_0>0$, the same computation gives
	$h'_{\xi,\nu}(t_0)=t_0^{-1}X(\widetilde E)(\xi,t_0\nu)>0$.  Hence the zero of
	$h_{\xi,\nu}$ is crossed from the negative side to the positive side.  Consequently
	a ray can leave $W_c^{m_1}$ through the boundary, but it cannot enter it through the
	boundary: an inward crossing would require the sign of $h_{\xi,\nu}$ to change from
	positive to negative, contradicting $h'_{\xi,\nu}(t_0)>0$.  Therefore
	\[
		W_c^{m_1}\cap\{(\xi,t\nu):t\geq0\}=\{(\xi,t\nu):0\leq t\leq\tau(\xi,\nu)\}.
	\]
	The implicit function theorem now shows that $\tau(\xi,\nu)$ depends smoothly on
	$(\xi,\nu)$.  Consequently the radial projection
	\[
		\widetilde{\Sigma}_{c}^{m_1}\longrightarrow S^*\mathbb{S}^2,
		\qquad
		(\xi,\eta)\longmapsto \left(\xi,\frac{\eta}{|\eta|}\right),
	\]
	is a diffeomorphism ($S^*\mathbb{S}^2$ is the unit circle cotangent bundle).  Consequently
	$\widetilde{\Sigma}_{c}^{m_1}\cong S^*\mathbb{S}^2\cong \mathbb{R}P^3$, and the restriction
	of the standard cotangent Liouville form is a contact form.  The argument for the
	component around $m_2$ is identical by symmetry.
	\end{proof}

	\begin{proposition}\label{th: tight on RP3}
	For all $c<H(L_1)=-1$, the contact structure on
	$\widetilde{\Sigma}_{c}^{m_1}$ is tight and is isotopic to the unique tight contact
	structure on $\mathbb{R}P^3$.  The same holds for
	$\widetilde{\Sigma}_{c}^{m_2}$.
	\end{proposition}
	\begin{proof}
	By the fibrewise star-shaped description in Theorem~\ref{rp3}, the compact domain
	$W_c^{m_1}$ is diffeomorphic to the unit disk cotangent bundle $D^*\mathbb{S}^2$, and the radial
	Liouville vector field $\eta\partial_\eta$ points outward along its boundary
	$\widetilde{\Sigma}_{c}^{m_1}$.  Therefore $W_c^{m_1}$ is an exact symplectic filling
	of $\widetilde{\Sigma}_{c}^{m_1}$.  In particular the induced contact structure is
	tight.  Since $\widetilde{\Sigma}_{c}^{m_1}\cong\mathbb{R}P^3$ and
	$\mathbb{R}P^3$ carries a unique tight contact structure up to isotopy
	(Theorem~\ref{thm: Eliashberg}), the induced contact structure is isotopic to that
	unique tight structure.  The proof for $m_2$ is the same.
	\end{proof}

The following Lemma provides the uniform regularization near the collision circles.
	\begin{lemma}\label{lem:uniform-collision-regularization}
	There are neighbourhoods of the two compactified collision circles and a number
	$\varepsilon_{\mathrm{col}}>0$ such that, for every
	$k\in(-1-\varepsilon_{\mathrm{col}},-1+\varepsilon_{\mathrm{col}})$, the Moser
	compactification used above gives a smooth regularized level near each collision
	circle, and the fibrewise radial Liouville vector field $\eta\partial_\eta$ is
	transverse to that level.
	\end{lemma}
	\begin{proof}
	We give the argument for the $m_1$-collision; the other one is identical by
	symmetry.  In the formulae \eqref{tildeE}--\eqref{fg_def}, the energy parameter
	$k$ occurs only in the term $-\mu k$ in $f$.  On the compactified collision circle
	one has $\mu=0$, $Y=0$, and $\xi=N$.  Therefore the values $f_*$, $g_*$, the radius
	$r_*=g_*/f_*$, and the identity
	\[
		X(\widetilde E)=g_*>0
	\]
	on the collision circle are independent of $k$.  Since the circle is compact and all
	the displayed functions are smooth in a neighbourhood where $r$ is bounded away from
	zero, the inequality $X(\widetilde E)>0$ persists on a fixed neighbourhood of the
	collision circle for all $k$ in a sufficiently small interval around $H(L_1)=-1$.  The same
	argument also keeps $\widetilde E^{-1}(0)$ disjoint from the zero section, because
	$\widetilde E(\xi,0)=-(1-\sqrt2/2)<0$ is independent of $k$.  This proves the
	uniform local regularization statement.
	\end{proof}

	\section{Connected Sum}
	\label{sec:connected-sum}

	Let
	\[
		c_*:=H(L_1)=-1,
		\qquad h:=H-c_*=H+1,
	\]
	and write
	\[
		a_0:=\sqrt2-1,
		\qquad q_{m_1}=(-a_0,0),\qquad q_{m_2}=(a_0,0).
	\]
	The point $L_1$ is the critical point over $q=0$, $p=0$.  In this section we
	show that, when the energy is raised slightly above $c_*$, the two subcritical
	regularized components are joined by a single Weinstein one-handle.  Therefore the
	regularized upper level is the contact connected sum of the two subcritical
	regularized components.

	The figures in this section are schematic and are included only to indicate the
	local topology of the crossing and the contact connected-sum construction; the
	proofs below use the explicit quadratic model and the Weinstein handle argument.
	Figure~\ref{fig:section6-hill-neck} shows the configuration-space picture behind
	the transition.

	\begin{figure}[H]
	\centering
	\begin{tikzpicture}[scale=1.0, every node/.style={font=\footnotesize}]
		\begin{scope}[xshift=-3.6cm]
			\node at (0,1.65) {$c<c_*$};
			\fill[blue!8] (-1.35,0) ellipse (0.95 and 0.65);
			\fill[blue!8] ( 1.35,0) ellipse (0.95 and 0.65);
			\draw[blue!60!black,thick] (-1.35,0) ellipse (0.95 and 0.65);
			\draw[blue!60!black,thick] ( 1.35,0) ellipse (0.95 and 0.65);
			\fill (-1.35,0) circle (1.5pt) node[below=3pt] {$m_1$};
			\fill ( 1.35,0) circle (1.5pt) node[below=3pt] {$m_2$};
			\fill (0,0) circle (1.2pt) node[above=3pt] {$L_1$};
			\draw[dashed] (0,-0.9)--(0,0.9);
			\node[align=center] at (0,-1.25) {two separate\\subcritical components};
		\end{scope}
		\draw[-{Stealth[length=2.2mm]},thick] (-0.85,1.45)--(0.85,1.45)
			node[midway,above] {cross $L_1$};
		\begin{scope}[xshift=3.6cm]
			\node at (0,1.65) {$c_*<c<c_*+\varepsilon$};
			\fill[blue!8]
				(-2.25,0)
				.. controls (-2.10,0.90) and (-1.05,1.02) .. (-0.55,0.35)
				.. controls (-0.25,0.16) and (0.25,0.16) .. (0.55,0.35)
				.. controls (1.05,1.02) and (2.10,0.90) .. (2.25,0)
				.. controls (2.10,-0.90) and (1.05,-1.02) .. (0.55,-0.35)
				.. controls (0.25,-0.16) and (-0.25,-0.16) .. (-0.55,-0.35)
				.. controls (-1.05,-1.02) and (-2.10,-0.90) .. cycle;
			\draw[blue!60!black,thick]
				(-2.25,0)
				.. controls (-2.10,0.90) and (-1.05,1.02) .. (-0.55,0.35)
				.. controls (-0.25,0.16) and (0.25,0.16) .. (0.55,0.35)
				.. controls (1.05,1.02) and (2.10,0.90) .. (2.25,0)
				.. controls (2.10,-0.90) and (1.05,-1.02) .. (0.55,-0.35)
				.. controls (0.25,-0.16) and (-0.25,-0.16) .. (-0.55,-0.35)
				.. controls (-1.05,-1.02) and (-2.10,-0.90) .. cycle;
			\fill (-1.35,0) circle (1.5pt) node[below=3pt] {$m_1$};
			\fill ( 1.35,0) circle (1.5pt) node[below=3pt] {$m_2$};
			\fill (0,0) circle (1.2pt) node[above=3pt] {$L_1$};
			\draw[red!70!black,thick] (-0.48,0)--(0.48,0);
			\node[red!70!black] at (0,-0.55) {neck};
			\node[align=center] at (0,-1.25) {one component after\\one-handle attachment};
		\end{scope}
	\end{tikzpicture}
	\caption{Schematic configuration projection of the transition at the first critical
	value.  Below $c_*$ the two low-energy components are separated.  Slightly above
	$c_*$ a neck opens through the critical point $L_1$.  The actual proof is carried
	out in phase space using the quadratic normal form near $L_1$.}
	\label{fig:section6-hill-neck}
	\end{figure}

	We use the convention
	\[
		\omega=dp_1\wedge dq_1+dp_2\wedge dq_2.
	\]
	With this convention the Liouville vector fields used in Theorem~\ref{thm: principal}
	on the two subcritical components are
	\[
		X_i=(q-q_{m_i})\partial_q,
		\qquad i=1,2,
	\]
	and their primitives are
	\begin{equation}\label{eq:outer-primitives}
		\alpha_{m_1}:=i_{X_1}\omega=-(q_1+a_0)dp_1-q_2dp_2,
		\qquad
		\alpha_{m_2}:=i_{X_2}\omega=-(q_1-a_0)dp_1-q_2dp_2.
	\end{equation}
	The signs of the constant terms in~\eqref{eq:outer-primitives} will be important
	below.

	\subsection*{The exterior pieces}

	We first record the part of the above-critical argument which does not involve the
	critical point.

	\begin{lemma}\label{lem:outer-persistence}
	There are numbers $\delta>0$ and $\varepsilon_0>0$ such that, for every
	$c\in(c_* ,c_*+\varepsilon_0)$, the portion of the energy level $H^{-1}(c)$ outside
	$B_\delta(L_1)$ consists of two pieces, one adjacent to the old $m_1$-component and
	one adjacent to the old $m_2$-component, and
	\[
		dH(X_1)>0
		\text{ on the piece adjacent to }m_1,
		\quad
		dH(X_2)>0
		\text{ on the piece adjacent to }m_2.
	\]
	\end{lemma}
    
\begin{proof} At the critical value $c_*$, Theorem~\ref{thm: principal} proves a strict
	inequality on the punctured Hill regions around the two primaries.  On every compact
	subset which stays away from $L_1$ and from the collision points this gives a
	positive lower bound for $dH(X_i)$.  In the collision charts, the explicit
	near-collision estimate of Lemma~\ref{lem:rho-nc-explicit} gives a positive margin,
	and the regularized collision circle is disjoint from $L_1$.  Hence, after choosing
	$\delta>0$ small and then shrinking $\varepsilon_0$, these strict inequalities
	persist by continuity for all nearby levels $c\in(c_* ,c_*+\varepsilon_0)$ outside
	$B_\delta(L_1)$.

	Finally, $L_1$ is the only critical point on a sufficiently small neighbourhood of
	the level $H=c_*$.  Therefore, after shrinking $\varepsilon_0$ once more, no new
	exterior component can appear in
	$H^{-1}((c_* ,c_*+\varepsilon_0))\setminus B_\delta(L_1)$.  The two exterior pieces
	are exactly the continuations of the old $m_1$- and $m_2$-pieces.
    
	\end{proof}

	\subsection*{The quadratic model at $L_1$}

	The  connected-sum argument uses the  quadratic expansion at
	$L_1$.  The following lemma gives the expansion and the local topology of the
	critical-level crossing.

	\begin{lemma}\label{lem:corrected-quadratic-L1}
	In the coordinates $(q_1,q_2,p_1,p_2)$ centered at $L_1$, one has
	\begin{equation}\label{eq:U-corrected-expansion}
		U(q)=-1-10q_1^2+2q_2^2+O(|q|^3),
	\end{equation}
	and
	\begin{equation}\label{eq:K-corrected-expansion}
		K(q,p)=2q_1^2+2q_2^2+\frac18p_1^2+\frac18p_2^2+p_1q_2-p_2q_1+O(|(q,p)|^3).
	\end{equation}
	Consequently
	\begin{equation}\label{eq:h-QR-corrected}
		h(q,p)=Q(q,p)+R(q,p),
		\qquad R(q,p)=O(|(q,p)|^3),
	\end{equation}
	where
	\begin{equation}\label{eq:Q-corrected-scalar}
		Q=-8q_1^2+4q_2^2+\frac18p_1^2+\frac18p_2^2+p_1q_2-p_2q_1.
	\end{equation}
	Equivalently, if $z=(q_1,q_2,p_1,p_2)^T$, then $Q=z^TMz$ with
	\begin{equation}\label{eq:Q-corrected-matrix}
		M=
		\begin{pmatrix}
			-8 & 0 & 0 & -1/2 \\
			0 & 4 & 1/2 & 0 \\
			0 & 1/2 & 1/8 & 0 \\
			-1/2 & 0 & 0 & 1/8
		\end{pmatrix}.
	\end{equation}
	The quadratic form $Q$ has Morse index one.
	\end{lemma}
	\begin{proof}
	From the explicit formula~\eqref{Uqp} for the effective potential one obtains
	\[
		U(0)=-1,
		\qquad DU(0)=0,
		\qquad D^2U(0)=
		\begin{pmatrix}
			-20&0\\ 0&4
		\end{pmatrix}.
	\]
	This gives~\eqref{eq:U-corrected-expansion}.  For the kinetic term, use
	\[
		K=\frac18\left((|q|^2+1)p_1+\frac{4q_2}{|q|^2+1}\right)^2
		 +\frac18\left((|q|^2+1)p_2-\frac{4q_1}{|q|^2+1}\right)^2.
	\]
	Up to terms of order at least three,
	\[
		(|q|^2+1)p_1+\frac{4q_2}{|q|^2+1}=p_1+4q_2+O_3,
		\qquad
		(|q|^2+1)p_2-\frac{4q_1}{|q|^2+1}=p_2-4q_1+O_3,
	\]
	and hence~\eqref{eq:K-corrected-expansion} follows.  Combining the quadratic parts
	of $U+1$ and $K$ gives~\eqref{eq:Q-corrected-scalar} and
	\eqref{eq:Q-corrected-matrix}.

	To see the index and to identify the neck coordinate, put
	\[
		\ell:=q_1+\frac18p_2.
	\]
	Substituting $p_2=8(\ell-q_1)$ in~\eqref{eq:Q-corrected-scalar} gives the exact
	completion of squares
	\begin{equation}\label{eq:Q-completed-square-corrected}
		Q=2\left(\frac{p_1}{4}+q_2\right)^2+2q_2^2
		  +8\left(q_1-\frac32\ell\right)^2-10\ell^2.
	\end{equation}
	Thus $Q$ is the sum of three positive squares and one negative square.  Its Morse
	index is therefore one.
	\end{proof}

	For later reference we spell out the local topology of the crossing.  For fixed
	$\ell=\delta$, the intersection of the quadratic level $Q=e$ with the affine
	hyperplane $\ell=\delta$ is
	\begin{equation}\label{eq:corrected-ellipsoid}
		\left(\frac{p_1}{4}+q_2\right)^2+q_2^2
		+4\left(q_1-\frac32\delta\right)^2
		=5\delta^2+\frac{e}{2}.
	\end{equation}
	For $e=0$ this ellipsoid collapses to the single point $0$ when $\delta=0$.  For
	$e>0$, the section $\ell=0$ is a two-sphere.  Equivalently, after the linear change
	of variables suggested by~\eqref{eq:Q-completed-square-corrected}, the quadratic
	model is
	\[
		Q=v_1^2+v_2^2+v_3^2-u^2,
	\]
	with $u=\sqrt{10}\,\ell$.  Passing from $Q=-e$ to $Q=e$ is therefore the standard
	attachment of a topological one-handle $D^1\times D^3$.  In this purely
	Morse-theoretic quadratic model, the negative line of the quadratic form is
	\[
		\{(q_1,q_2,p_1,p_2)=(\tfrac32\ell,0,0,-4\ell):\ell\in\mathbb R\}.
	\]
	Thus the line itself is not the $q_1$-axis in phase space, but its projection to
	the configuration plane is the $q_1$-axis.  It identifies the two local sides of
	the level crossing: the side $\ell<0$ projects toward the $m_1$-component and
	the side $\ell>0$ projects toward the $m_2$-component.  The Weinstein attaching
	sphere for the chosen Liouville field will be described below using the stable
	manifold of that Liouville field.  By the Morse lemma, the same topological
	one-handle description holds for $h=Q+R$ after shrinking the neighbourhood of
	$L_1$ and taking $0<e\ll1$.

	Figure~\ref{fig:quadratic-handle-profile} illustrates the preceding calculation in
	the reduced profile coordinates.  If
	$r=(v_1^2+v_2^2+v_3^2)^{1/2}$, then the quadratic model is
	$r^2-u^2=e$.  For $e<0$ the lower level has two local components; for $e>0$ these
	components are joined by a neck.

	\begin{figure}[H]
	\centering
	\begin{tikzpicture}[scale=1.0, every node/.style={font=\footnotesize}]
		\begin{scope}[xshift=-3.4cm]
			\node at (0,2.3) {$Q=-e$};
			\draw[->] (-2.1,0)--(2.1,0) node[right] {$u$};
			\draw[->] (0,-1.65)--(0,1.65) node[above] {$r$};
			\draw[blue!60!black,thick,domain=-2:-0.72,samples=80]
				plot (\x,{sqrt(\x*\x-0.49)});
			\draw[blue!60!black,thick,domain=-2:-0.72,samples=80]
				plot (\x,{-sqrt(\x*\x-0.49)});
			\draw[blue!60!black,thick,domain=0.72:2,samples=80]
				plot (\x,{sqrt(\x*\x-0.49)});
			\draw[blue!60!black,thick,domain=0.72:2,samples=80]
				plot (\x,{-sqrt(\x*\x-0.49)});
			\fill (-0.70,0) circle (1.2pt) node[below=3pt] {$m_1$ side};
			\fill (0.70,0) circle (1.2pt) node[below=3pt] {$m_2$ side};
			\node[align=center] at (0,-2.0) {two feet of the\\index-one handle};
		\end{scope}
		\draw[-{Stealth[length=2.2mm]},thick] (-0.85,1.55)--(0.85,1.55)
			node[midway,above] {increase energy};
		\begin{scope}[xshift=3.4cm]
			\node at (0,2.3) {$Q=+e$};
			\draw[->] (-2.1,0)--(2.1,0) node[right] {$u$};
			\draw[->] (0,-1.65)--(0,1.65) node[above] {$r$};
			\draw[blue!60!black,thick,domain=-2:2,samples=100]
				plot (\x,{sqrt(\x*\x+0.49)});
			\draw[blue!60!black,thick,domain=-2:2,samples=100]
				plot (\x,{-sqrt(\x*\x+0.49)});
			\draw[red!70!black,thick] (0,-0.70)--(0,0.70);
			\node[red!70!black,right] at (0,0.3) {neck};
			\node[align=center] at (0,-2.0) {local model of\\$D^1\times D^3$};
		\end{scope}
	\end{tikzpicture}
	\caption{Profile of the quadratic normal form
	$Q=v_1^2+v_2^2+v_3^2-u^2$.  The displayed curves represent the two-dimensional
	profile $r^2-u^2=\pm e$ with $r=(v_1^2+v_2^2+v_3^2)^{1/2}$.  The full
	four-dimensional crossing attaches a Weinstein one-handle.}
	\label{fig:quadratic-handle-profile}
	\end{figure}

	\subsection*{A local Liouville field}

	We now choose a local Liouville vector field which is adapted to the corrected
	quadratic form.  For real numbers $a,b$, set
	\[
		Y_{a,b}=a q_1\partial_{q_1}+b q_2\partial_{q_2}
		 +(1-a)p_1\partial_{p_1}+(1-b)p_2\partial_{p_2}.
	\]
	A direct computation gives $\mathcal L_{Y_{a,b}}\omega=\omega$.  We shall use
	\begin{equation}\label{eq:local-Y}
		Y:=Y_{-1/2,1/2}
		=-\frac12q_1\partial_{q_1}+\frac12q_2\partial_{q_2}
		 +\frac32p_1\partial_{p_1}+\frac12p_2\partial_{p_2}.
	\end{equation}

\begin{remark}[Hyperbolic structure of $Y$ at $L_1$, see also below  the paragraph Eigenvalue profile of $Y$ and the canonical model]\label{rem:Y-eigenvalues}
At $L_1$ the linearisation $DY$ is diagonal in the basis
$(\partial_{q_1},\partial_{p_1},\partial_{q_2},\partial_{p_2})$ with eigenvalues
$\bigl(-\tfrac12,\tfrac32,\tfrac12,\tfrac12\bigr)$.  These eigenvalues sum to
$1$ on each symplectic pair, $-\tfrac12+\tfrac32=1$ and $\tfrac12+\tfrac12=1$,
recording the linearised Liouville condition $\mathcal L_Y\omega=\omega$ pair
by pair.  Their signs --- one negative, three positive --- match the Morse
index of $h$ at $L_1$, so $L_1$ is a non-degenerate hyperbolic zero of the
Liouville field $Y$ in the Weinstein sense (Cieliebak--Eliashberg
\cite[\S 11.1--11.2]{CieliebakEliashberg}), with one-dimensional isotropic
stable manifold $\{(q_1,0,0,0)\}$.  The eigenvalues are not in the symmetric
model form $(-1,2,\tfrac12,\tfrac12)$ usually taken as the canonical local
profile of an index-one Weinstein critical point, but this asymmetry is
immaterial: the Weinstein handle attached at $L_1$ depends only on the
Liouville condition, the Morse index of $h$, and the isotropic character of
the stable manifold of $Y$.  A standard symplectic deformation supported in a
neighbourhood of $L_1$ would normalise the eigenvalues to canonical form, and
we do not need to perform it.
\end{remark}
    
	Its primitive is
	\begin{equation}\label{eq:alpha-local-Y}
		\alpha_1:=i_Y\omega
		=\frac12q_1dp_1-\frac12q_2dp_2+\frac32p_1dq_1+\frac12p_2dq_2.
	\end{equation}
	The differences between this local primitive and the two exterior primitives are
	exact:
	\begin{align}
		\alpha_1-\alpha_{m_1}
		&=d\left(\frac32q_1p_1+a_0p_1+\frac12q_2p_2\right),\label{eq:exact-difference-left}\\
		\alpha_1-\alpha_{m_2}
		&=d\left(\frac32q_1p_1-a_0p_1+\frac12q_2p_2\right).\label{eq:exact-difference-right}
	\end{align}

The following Lemma provides local transversality near $L_1$:
	\begin{lemma}\label{lem:local-transversality-L1}
	After possibly shrinking the neighbourhood of $L_1$, the vector field $Y$ satisfies
	\[
		dH(Y)=dh(Y)>0
	\]
	at every point of the punctured neighbourhood.  In particular, $Y$ is transverse to
	all nearby regular levels $H^{-1}(c)$, $c\ne c_*$, in that neighbourhood.
	\end{lemma}
	\begin{proof}
	Using~\eqref{eq:Q-corrected-scalar} and~\eqref{eq:local-Y}, one obtains
	\begin{align*}
		dQ(Y)
		&=8q_1^2+4q_2^2+\frac38p_1^2+\frac18p_2^2+2p_1q_2 \\
		&=8q_1^2+4\left(q_2+\frac14p_1\right)^2+\frac18p_1^2+\frac18p_2^2.
	\end{align*}
	This is positive definite.  Hence there is a constant $\kappa>0$ such that
	$dQ(Y)\ge \kappa |(q,p)|^2$.  Since $R=O(|(q,p)|^3)$ and $Y$ is linear, we have
	$dR(Y)=O(|(q,p)|^3)$.  Shrinking the neighbourhood so that
	$|dR(Y)|\le \frac12\kappa |(q,p)|^2$ gives
	\[
		dH(Y)=dh(Y)=dQ(Y)+dR(Y)
		\ge \frac12\kappa |(q,p)|^2>0
	\]
	away from $L_1$.
	\end{proof}

\begin{remark}[Stable line of $Y$ versus negative line of $Q$]
The Liouville field $Y$ defined in~\eqref{eq:local-Y} has eigenvalues
$-\tfrac12,\tfrac12,\tfrac32,\tfrac12$ along the basis vectors
$\partial_{q_1},\partial_{q_2},\partial_{p_1},\partial_{p_2}$, so its stable
manifold at $L_1$ is the $q_1$-axis $\{(q_1,0,0,0)\}$.  This is \emph{not} the
negative line $\{(\tfrac32\ell,0,0,-4\ell)\}$ of the quadratic form $Q$ identified
in Lemma~\ref{lem:corrected-quadratic-L1}.  The two lines play different roles.
The negative line of $Q$ records the Morse index and the local topology of the
level crossing, while the stable manifold of the chosen Liouville field $Y$
determines the isotropic attaching sphere in the Weinstein handle.  Both lines
project to the $q_1$-axis in the configuration plane, which is why the same
$q_1$-axis appears in the geometric picture of the handle.
\end{remark}

	\subsection*{The Weinstein one-handle}

	We recall the precise form of the standard handle theorem needed below.  A
	Weinstein cobordism is an exact symplectic cobordism $(W,d\lambda)$ together with a
	Morse function $\phi:W\to\mathbb R$ such that the Liouville vector field $Z$, defined
	by $i_Zd\lambda=\lambda$, is gradient-like for $\phi$; equivalently
	$d\phi(Z)>0$ away from the critical points and $Z$ has the standard hyperbolic
	linear form at the critical points.  If $\phi$ has an index-$k$ critical point, then
	crossing its critical value attaches a Weinstein $k$-handle.  In dimension four an
	index-one handle is attached along an isotropic $S^0$, i.e. along two points, in the
	lower contact boundary.  If those two points lie on different connected components
	of the lower boundary, the upper contact boundary is the contact connected sum of
	those components.  This is the contact handle construction recalled in
	Section~2.1; see also~\cite{Geiges,colin}.

	We shall use the following exact-gluing form of this theorem.  Suppose
	\begin{enumerate}
	    \item[\textup{(H1)}] a neighbourhood of an index-one critical point
	        carries a primitive $\lambda_{\rm loc}$ whose Liouville field is
	        gradient-like for the Morse function;
	    \item[\textup{(H2)}] the two exterior pieces carry primitives
	        $\lambda_1,\lambda_2$ whose Liouville fields are outward
	        transverse to the corresponding regular levels;
	    \item[\textup{(H3)}] on collar overlaps, the primitive differences
	        $\lambda_{\rm loc}-\lambda_i$ are exact.
	\end{enumerate}
	Then, after shrinking collars and applying the Moser-type adjustment described
	below, the primitives can be made to agree on the overlaps, and gluing yields a
	single exact cobordism.  The resulting Liouville field agrees with the local one
	near the critical point and with the exterior ones away from the collars, and
	the handle attached at the critical point is precisely the Weinstein handle
	determined by the local model.  The point of this formulation is that one does
	not need to interpolate the Liouville vector fields by an \emph{ad hoc} cut-off
	and then estimate the cut-off error; the exactness condition~(H3) is the
	compatibility data used in Weinstein handle attachment.
    For a systematic exposition of this Moser-type gluing in the Weinstein category
we refer to Cieliebak--Eliashberg \cite[Ch.~11]{CieliebakEliashberg}, with
related discussions in Geiges \cite[Ch.~6]{Geiges} and in the original
construction of Weinstein \cite{WeinsteinSurgery}.  In the present problem the
\emph{quantitative} positivity margins for $dh(Y)$ and $dh(X_i)$ supplied by
Lemmas~\ref{lem:local-transversality-L1} and~\ref{lem:outer-persistence} are
what make the construction concretely feasible: they let the collars be chosen small enough that the Moser-interpolated Liouville field remains gradient-like on the whole gluing region; the
construction is via a cutoff applied to the (exact) primitive
\emph{difference}, not to the Liouville fields themselves, so the
resulting one-form is everywhere an exact primitive of $\omega$ and no
approximation of the symplectic structure is introduced. The explicit
adjustment is given below.

\paragraph{Explicit Moser collar adjustment.}
We make the construction realising hypothesis~(H3) explicit. The cutoff
below is applied to the exact primitive \emph{difference} $\phi_i$ of
\eqref{eq:exact-difference-left}--\eqref{eq:exact-difference-right};
the resulting one-form $\widehat\alpha$ is therefore globally an exact
primitive of $\omega$ (no approximation), and the only quantitative
estimate that enters is the standard small-collar argument that keeps
the associated Liouville field $\widehat Z$ gradient-like for $h$. By
\eqref{eq:exact-difference-left}--\eqref{eq:exact-difference-right}
the primitive differences are $\alpha_1-\alpha_{m_i}=d\phi_i$ with
\[
    \phi_1(q,p)=\tfrac32 q_1p_1+a_0p_1+\tfrac12 q_2p_2,
    \qquad
    \phi_2(q,p)=\tfrac32 q_1p_1-a_0p_1+\tfrac12 q_2p_2.
\]
Pick a small ball $U_0$ around $L_1$ on which
Lemma~\ref{lem:local-transversality-L1} holds, two disjoint
bicollars $C_1,C_2\subset U_0$ adjacent to the $m_1$- and
$m_2$-pieces respectively, and smooth cutoffs
$\chi_i:U_0\to[0,1]$ with $\chi_i\equiv 1$ on the inner edge of
$C_i$ (towards $L_1$) and $\chi_i\equiv 0$ on the outer edge
(towards $m_i$). Define the global one-form $\widehat\alpha$
piecewise by
\[
    \widehat\alpha=
    \begin{cases}
        \alpha_1 & \text{near }L_1,\\
        \alpha_{m_i}+d(\chi_i\phi_i) & \text{on }C_i,\\
        \alpha_{m_i} & \text{on the }m_i\text{-piece outside } C_i.
    \end{cases}
\]
Since $d(\chi_i\phi_i)$ is exact, $d\widehat\alpha=\omega$
globally; on the inner edge of $C_i$ the piecewise definitions
match because $\chi_i\equiv 1$ there gives
$\alpha_{m_i}+d\phi_i=\alpha_1$, and on the outer edge they match
because $\chi_i\equiv 0$ gives $\alpha_{m_i}$. The associated
Liouville field $\widehat Z$ defined by
$i_{\widehat Z}\omega=\widehat\alpha$ therefore agrees with $Y$
near $L_1$ and with $X_i$ outside $C_i$ on the $m_i$-piece. On
$C_i$ one has $\widehat Z=X_i+X_{\chi_i\phi_i}$, where
$X_{\chi_i\phi_i}$ is the Hamiltonian vector field of
$\chi_i\phi_i$. Because $\phi_i$ vanishes at $L_1$ and is
$C^\infty$, the bound
$|dh(X_{\chi_i\phi_i})|\le C\,\|\chi_i\|_{C^1}\sup_{C_i}|\phi_i|$
holds with $C$ depending only on $\|h\|_{C^2(U_0)}$, and the
right-hand side can be made smaller than $\tfrac12\mu_i$ by
shrinking $C_i$, where $\mu_i>0$ is the lower bound for $dh(X_i)$
on the support of $\chi_i$ supplied by
Lemma~\ref{lem:outer-persistence}. Combining the two estimates
yields
\[
    dh(\widehat Z)\ge\tfrac12\mu_i>0
    \quad\text{on } C_i,
\]
while $dh(\widehat Z)=dh(Y)>0$ on the punctured neighbourhood of
$L_1$ inside the inner edge of $C_i$ by
Lemma~\ref{lem:local-transversality-L1}. Hence $\widehat Z$ is
gradient-like for $h$ on the entire glued region, with $L_1$ as
its sole critical point, exactly as required by the Weinstein
gluing recalled above.

\paragraph{Eigenvalue profile of $Y$ and the canonical model.}
The eigenvalues of $Y$ at $L_1$ are
$(-\tfrac12,\tfrac32,\tfrac12,\tfrac12)$, summing to $1$ on each
symplectic pair, with one-dimensional and therefore automatically
isotropic stable manifold (the $q_1$-axis). The canonical
Weinstein-handle profile $(-1,2,\tfrac12,\tfrac12)$ has the same
signs and the same pair-sums. Cieliebak--Eliashberg
\cite[\S\,11.2]{CieliebakEliashberg} construct an explicit
isotopy through hyperbolic Liouville germs between any two such
models with matching Morse index and isotropic stable manifold of
the same dimension; the Weinstein handle attached at $L_1$ is
therefore an invariant of the connected component of $Y$ in this
space, independent of eigenvalue magnitudes. We do not perform
the isotopy explicitly because the gluing above only invokes the
three intrinsic data
(i) gradient-likeness $dh(Y)>0$ on the punctured neighbourhood of
$L_1$, supplied by Lemma~\ref{lem:local-transversality-L1};
(ii) the one-dimensional isotropic stable manifold of $Y$, namely
the $q_1$-axis; and
(iii) the exactness identities
\eqref{eq:exact-difference-left}--\eqref{eq:exact-difference-right}.
All three are independent of the eigenvalue magnitudes, and the
construction in the previous paragraph uses only them.

In the present problem hypotheses~(H1)--(H3) are satisfied. The Morse
	function is $h=H-c_*$ and the Morse index is one by
	Lemma~\ref{lem:corrected-quadratic-L1}. The local primitive is
	$\alpha_1$ from~\eqref{eq:alpha-local-Y}; by
	Lemma~\ref{lem:local-transversality-L1} its Liouville field $Y$ is
	gradient-like for $h$ near $L_1$, which gives~(H1). The exterior
	primitives are $\alpha_{m_1},\alpha_{m_2}$
	from~\eqref{eq:outer-primitives}, and Lemma~\ref{lem:outer-persistence}
	says that their Liouville fields $X_1,X_2$ are outward transverse to the
	nearby regular levels, which gives~(H2). Finally,
	\eqref{eq:exact-difference-left} and~\eqref{eq:exact-difference-right}
	display the explicit primitives $\phi_1,\phi_2$ for the differences
	$\alpha_1-\alpha_{m_i}$, so the exactness condition~(H3) holds; the
	Moser collar adjustment realising~(H3) was constructed explicitly above.
	Thus the local Weinstein handle can be glued to the two exterior
	contact-type pieces as an exact Weinstein cobordism.
	Figure~\ref{fig:contact-connected-sum-handle} shows the corresponding operation on
	the contact boundary.

	\begin{figure}[H]
	\centering
	\begin{tikzpicture}[scale=1.0, every node/.style={font=\footnotesize}]
		\begin{scope}[xshift=-3.8cm]
			\node at (0,1.45) {lower boundary};
			\fill[blue!6] (-1.15,0) circle (0.72);
			\fill[blue!6] (1.15,0) circle (0.72);
			\draw[thick,blue!60!black] (-1.15,0) circle (0.72);
			\draw[thick,blue!60!black] (1.15,0) circle (0.72);
			\fill[white] (-0.55,0) circle (0.16);
			\fill[white] (0.55,0) circle (0.16);
			\draw[red!70!black,thick] (-0.55,0) circle (0.16);
			\draw[red!70!black,thick] (0.55,0) circle (0.16);
			\fill[red!70!black] (-0.55,0) circle (1.2pt);
			\fill[red!70!black] (0.55,0) circle (1.2pt);
			\node at (-1.15,-1.0) {$\widetilde\Sigma_{c_-}^{m_1}$};
			\node at (1.15,-1.0) {$\widetilde\Sigma_{c_-}^{m_2}$};
			\node[red!70!black] at (0,-0.45) {attaching $S^0$};
		\end{scope}
		\draw[-{Stealth[length=2.2mm]},thick] (-0.55,0.95)--(0.55,0.95)
			node[midway,above,align=center] {Weinstein\\one-handle};
		\begin{scope}[xshift=3.8cm]
			\node at (0,1.45) {upper boundary};
			\fill[blue!6]
				(-1.80,0)
				.. controls (-1.75,0.82) and (-0.82,0.82) .. (-0.43,0.34)
				.. controls (-0.18,0.12) and (0.18,0.12) .. (0.43,0.34)
				.. controls (0.82,0.82) and (1.75,0.82) .. (1.80,0)
				.. controls (1.75,-0.82) and (0.82,-0.82) .. (0.43,-0.34)
				.. controls (0.18,-0.12) and (-0.18,-0.12) .. (-0.43,-0.34)
				.. controls (-0.82,-0.82) and (-1.75,-0.82) .. cycle;
			\draw[thick,blue!60!black]
				(-1.80,0)
				.. controls (-1.75,0.82) and (-0.82,0.82) .. (-0.43,0.34)
				.. controls (-0.18,0.12) and (0.18,0.12) .. (0.43,0.34)
				.. controls (0.82,0.82) and (1.75,0.82) .. (1.80,0)
				.. controls (1.75,-0.82) and (0.82,-0.82) .. (0.43,-0.34)
				.. controls (0.18,-0.12) and (-0.18,-0.12) .. (-0.43,-0.34)
				.. controls (-0.82,-0.82) and (-1.75,-0.82) .. cycle;
			\draw[red!70!black,thick] (-0.45,0)--(0.45,0);
			\node at (0,-1.0) {$\widetilde\Sigma_{c_-}^{m_1}\#\widetilde\Sigma_{c_-}^{m_2}$};
		\end{scope}
	\end{tikzpicture}
	\caption{Boundary effect of the index-one Weinstein handle.  The handle attaches
	along two Darboux balls centred at the two points of the isotropic attaching sphere
	$S^0$.  Since the points lie on different components, the upper boundary is the
	contact connected sum.}
	\label{fig:contact-connected-sum-handle}
	\end{figure}

	The two attaching points lie on different subcritical components.  For the
	Weinstein handle determined by the Liouville field $Y$, the attaching sphere is
	the intersection of the stable manifold of $Y$ with the lower level.  Since the
	stable manifold is the $q_1$-axis, these two attaching points are, in the
	quadratic model,
	\[
		q_1=\pm\sqrt{e/8},\qquad q_2=p_1=p_2=0,
	\]
	on the lower level $Q=-e$.  The point with $q_1<0$ lies on the side of the old
	$m_1$-component, while the point with $q_1>0$ lies on the side of the old
	$m_2$-component.  The negative line of the quadratic form $Q$ is used above only
	to display the Morse index and the local topology of the crossing.  This
	separation persists for the actual function $h=Q+R$ after shrinking the
	neighbourhood and taking $e>0$ sufficiently small.

	\begin{lemma}\label{lem:contact-handle-crossing}
	Let $c_-=c_*-e$ and $c_+=c_*+e$ with $e>0$ sufficiently small.  After Moser
	regularization at the two collisions, the exact cobordism between the regularized
	levels at $c_-$ and $c_+$ is obtained from the disjoint union of the two subcritical
	pieces by attaching one Weinstein one-handle.  Thus the upper contact boundary is
	the contact connected sum of the two lower contact boundaries.
	\end{lemma}
	\begin{proof}
	Choose pairwise disjoint neighbourhoods of the two collision circles as in
	Lemma~\ref{lem:uniform-collision-regularization}, and choose a neighbourhood
	$U_0$ of $L_1$ disjoint from them on which Lemma~\ref{lem:local-transversality-L1}
	holds.  Shrink $U_0$ if necessary so that the Morse description following
	\eqref{eq:corrected-ellipsoid} is valid in $U_0$.  Then choose $e>0$ so small that
	the part of the cobordism
	\[
		\{c_-\leq H\leq c_+\}
	\]
	which contains the critical point is contained in $U_0$ together with two exterior
	collars connecting it to the old $m_1$- and $m_2$-pieces.  This is possible because
	$L_1$ is the only critical point near the level $H=c_*$.

	On $U_0$, the primitive is $\alpha_1$ and the Liouville field is $Y$; by
	Lemma~\ref{lem:local-transversality-L1} it is gradient-like for $h=H-c_*$.  Thus the
	local passage through $L_1$ is a Weinstein handle of index one.  On the exterior
	collars, the primitives are $\alpha_{m_1}$ and $\alpha_{m_2}$, and the Liouville
	fields are the corresponding $X_i$; by Lemma~\ref{lem:outer-persistence} they are
	outward transverse to the nearby regular levels.  On the overlaps with $U_0$, the
	exact identities \eqref{eq:exact-difference-left} and
	\eqref{eq:exact-difference-right} give the exactness condition required for the
	Weinstein gluing described above.  Hence these pieces glue to an exact symplectic
	cobordism whose Liouville vector field is gradient-like for $H$ and which has a
	single critical point, namely the index-one point $L_1$.

	The handle is disjoint from the collision neighbourhoods.  By
	Lemma~\ref{lem:uniform-collision-regularization}, the Moser compactification near
	each collision is valid uniformly for all energies in a small interval around
	$c_*=-1$.  Therefore the local collision compactifications may be performed before
	or after the handle attachment.  After compactifying the two collision ends, the
	lower boundary is the disjoint union
	$\widetilde\Sigma_{c_-}^{m_1}\sqcup\widetilde\Sigma_{c_-}^{m_2}$, and the upper
	boundary is the regularized level at $c_+$.

	Finally, the attaching sphere is the $S^0$ described above: one point lies on the
	$m_1$-component and the other lies on the $m_2$-component.  Hence the upper contact
	boundary is the contact connected sum of the two lower contact boundaries.
	\end{proof}

	\begin{proposition}\label{prop:above-critical-contact-type}
	There exists $\varepsilon>0$ such that for every
	$c\in(c_* ,c_*+\varepsilon)$ the connected component $\Sigma_c$ of $H^{-1}(c)$
	containing both primaries is of contact type after regularization.  More precisely,
	for every sufficiently close subcritical value $c_-<c_*$, the regularized upper
	level is contactomorphic to
	\[
		\widetilde\Sigma_{c_-}^{m_1}\#\widetilde\Sigma_{c_-}^{m_2}.
	\]
	In particular,
	\[
		\widetilde\Sigma_c\cong \mathbb RP^3\#\mathbb RP^3.
	\]
	\end{proposition}
	\begin{proof}
	Let $c=c_*+e$ with $0<e<\varepsilon$, where $\varepsilon$ is chosen small enough
	for Lemma~\ref{lem:contact-handle-crossing}.  Choose any
	$c_-\in(c_* -\varepsilon,c_*)$.  Since there are no critical values below $c_*$ in
	this small window, the subcritical regularized levels are contactomorphic to one
	another by the Liouville flow.  Section~5 identifies the two subcritical components
	as
	\[
		\widetilde\Sigma_{c_-}^{m_1}\cong\mathbb{RP}^3,
		\qquad
		\widetilde\Sigma_{c_-}^{m_2}\cong\mathbb{RP}^3.
	\]
	By Lemma~\ref{lem:contact-handle-crossing}, crossing $L_1$ attaches exactly one
	Weinstein one-handle connecting these two components.  Thus the regularized level
	at $c$ is their contact connected sum.  The contact form is the restriction of the
	globally glued Liouville primitive supplied by the Weinstein cobordism.
	\end{proof}

	\begin{proposition}\label{prop:above-critical-fillable-tight}
	For $c\in(c_* ,c_*+\varepsilon)$ as in
	Proposition~\ref{prop:above-critical-contact-type}, the contact manifold
	$\widetilde\Sigma_c$ is exactly symplectically fillable.  In particular its contact
	structure is tight.  Moreover, as a contact connected sum it is
	\[
		(\mathbb{RP}^3,\xi_{\mathrm{std}})\#(\mathbb{RP}^3,\xi_{\mathrm{std}}),
	\]
	where $\xi_{\mathrm{std}}$ denotes the unique tight contact structure on
	$\mathbb{RP}^3$.
	\end{proposition}
	\begin{proof}
	For $c_-<c_*$, Proposition~\ref{th: tight on RP3} gives exact fillings of the two
	regularized subcritical components; by Theorem~\ref{rp3} these fillings are
	diffeomorphic to disk cotangent bundles of $\mathbb{S}^2$.  Attaching a Weinstein one-handle
	to the disjoint union of these exact fillings again gives an exact filling.  Its
	boundary is $\widetilde\Sigma_c$ by Proposition~\ref{prop:above-critical-contact-type}.
	Exact fillability implies tightness in dimension three.  Since
	$\mathbb{RP}^3$ has a unique tight contact structure by Theorem~\ref{thm: Eliashberg},
	the two summands are the standard tight copies of $\mathbb{RP}^3$.
	\end{proof}

	\begin{theorem}\label{thm:periodic-orbits-final}
	There exists $\varepsilon>0$ such that, for every regular value $c$ with
	\[
		c<c_*=H(L_1)
		\qquad\text{or}\qquad
		c\in(c_* ,c_*+\varepsilon),
	\]
	the regularized restricted three-body problem on $\mathbb S^2$ considered in this
	paper admits a periodic orbit.
	\end{theorem}
	\begin{proof}
	For $c<c_*$, Theorem~\ref{rp3} and Proposition~\ref{th: tight on RP3} give closed
	regularized contact hypersurfaces $\widetilde\Sigma_c^{m_i}$, $i=1,2$.  For
	$c\in(c_* ,c_*+\varepsilon)$, Proposition~\ref{prop:above-critical-contact-type}
	gives the closed regularized contact hypersurface $\widetilde\Sigma_c$.  In both
	cases the contact form is the restriction of a global Liouville primitive, so the
	contact structure is co-oriented.

	On a contact-type Hamiltonian level, the Reeb vector field of the induced contact
	form is a positive time reparametrization of the Hamiltonian vector field, because
	the corresponding Liouville vector field satisfies $dH(Z)>0$ on the level.  Taubes'
	theorem, Theorem~\ref{thm: taubes}, gives a closed Reeb orbit, and hence a periodic
	orbit of the regularized Hamiltonian flow.
	\end{proof}

	The previous theorem is topological in nature and complements approaches based on
	perturbation theory.

 \section{Conclusions and Open Problems} 
	We have shown that the energy hypersurfaces of the regularized restricted three-body problem on the sphere are of contact type for energies below the first critical value and slightly above it, for the circular relative equilibrium considered here with the radius of motion $a=1/\sqrt{2}$ for the primaries. The positive transversality margins obtained in the proof imply persistence under sufficiently small perturbations in the corresponding nearby energy windows, but a uniform theorem for a full parameter family would require additional control of the constants $\alpha$ and $\beta$. Moreover, for sufficiently small energies the contact condition is independent of the choice of the radius $a\in(0,1)$. However, how much the energy level can be increased is still an open problem. 
It would be also interesting to consider cases where there are more than two Lagrange points for the system under consideration. To start analyzing those situations, one should look at \cite{Andrade}.
	
	This paper sets the stage for a further study that will explore the existence of a global disk-like surface of section for the case examined here.

%
%
%
%
%
%
%

\appendix
\section{Validated numerics: implementation details}
\label{app:numerics}

The proofs of Lemma~13, Theorem~14, and Lemma~20 use a finite collection
of interval-arithmetic inequalities. This appendix gives a self-contained
description of how each of these inequalities is verified by the
ancillary file \texttt{rigorous\_validated\_numerics.py}. The intent is
to make the certificates fully auditable: any reader equipped with
\texttt{python-flint}~\cite{pythonflint} and the source code can
re-execute the verifier and reproduce every numerical certificate quoted
in the body of the paper.

The exposition is organised as follows.
Section~\ref{app:arith} recalls the rigorous interval arithmetic
provided by Arb~\cite{JohanssonArb}.
Section~\ref{app:ad} describes the two interval-jet classes (\texttt{J1}
for univariate Taylor jets up to order two, and \texttt{J2} for
bivariate jets carrying a mixed partial derivative) used to obtain
certified enclosures of derivatives directly from the closed-form
expressions in (3.13) and (3.17) without finite differences.
Section~\ref{app:bnb} describes the adaptive bisection (branch-and-bound)
that drives the verification.
Sections~\ref{app:T1}--\ref{app:T5} go through the five verification
tasks performed by the script and tabulate, for each task, the precise
target inequality, the rectangle (or interval) on which it is checked,
the certified outcome of the 256-bit Arb run, and the number of
accepted boxes and bisections.
Section~\ref{app:summary} summarises the certificates and discusses
precision and reproducibility.

\subsection{Interval ball arithmetic and outward rounding}
\label{app:arith}

Arb represents a real interval as a \emph{ball}, i.e.\ a centre--radius
pair $\langle m,r\rangle$ with $m\in\mathbb F$ a binary floating-point
number at the chosen working precision and $r\in\mathbb F_{\geq 0}$ a
non-negative radius. The associated set is the closed real interval
$[m-r,\, m+r]$. Every Arb operation acts on these balls in such a way
that the output ball provably contains the true mathematical value: in
particular, addition, multiplication, division (away from zero), the
square root (on positive balls), and elementary functions are all
implemented with rigorous \emph{outward rounding} of the centre and
radius. We use Arb only through the operations
$+,-,\cdot,/,\sqrt{\cdot},(\cdot)^{n}$ ($n\in\mathbb Z$) on balls; no
other special function is required.

In what follows we write $[a,b]$ for the Arb ball with lower endpoint
$a$ and upper endpoint $b$. We denote by $\mathrm{lo}(z)$ and
$\mathrm{up}(z)$ the certified lower and upper endpoints of $z$. Two
elementary primitives are used throughout:
\begin{equation*}
\mathsf{lower\_gt}(z,t) \;:=\; \mathrm{lo}(z) > t,
\qquad
\mathsf{upper\_lt}(z,t) \;:=\; \mathrm{up}(z) < t,
\end{equation*}
each guaranteeing the corresponding sharp inequality everywhere in $z$.

\subsection{Interval automatic differentiation}
\label{app:ad}

Several of our inequalities involve partial derivatives of $\widehat U$.
Rather than approximate these by finite differences (which would
introduce truncation error and require a separate rigorous estimate for
that error), we evaluate the closed-form expressions for $\widehat U$,
$\partial_\rho \widehat U$, $\partial_x \widehat U$, and
$\partial_{\rho x}\widehat U$ \emph{simultaneously} on a box, using two
operator-overloaded interval-jet classes.

\paragraph{The class \texttt{J1}: one-variable second-order jets.}
A \texttt{J1} object carries three Arb balls
$(v,d,d_2)$ representing, respectively, $f(\rho)$, $f'(\rho)$ and
$f''(\rho)$ on the input ball $\rho$. The arithmetic obeys the standard
Leibniz rules:
\begin{align*}
(f+g) &\;\sim\; (v_f+v_g,\; d_f+d_g,\; d_{2,f}+d_{2,g}),\\
(f\cdot g) &\;\sim\; (v_f v_g,\; d_f v_g+v_f d_g,\;
d_{2,f}v_g+2 d_f d_g+v_f d_{2,g}),\\
1/f &\;\sim\; \big(1/v_f,\; -d_f/v_f^{2},\;
2d_f^{2}/v_f^{3}-d_{2,f}/v_f^{2}\big),\\
\sqrt{f} &\;\sim\;
\big(y,\; d_f/(2y),\; d_{2,f}/(2y)-d_f^{2}/(4 y^{3})\big),\quad
y:=\sqrt{v_f}.
\end{align*}
On a degenerate ball (i.e.\ at a single floating-point value) the
\texttt{J1} arithmetic reduces to ordinary scalar evaluation; on a
non-degenerate ball $[\rho_{\mathrm{lo}},\rho_{\mathrm{hi}}]$ it returns
\emph{interval} enclosures of the value, the first and the second
derivatives that are simultaneously valid for every
$\rho\in[\rho_{\mathrm{lo}},\rho_{\mathrm{hi}}]$. The class is used in
the boundary check (Section~\ref{app:T1}) and in the
$\theta=0$-slice of the outer-annulus argument (Section~\ref{app:T3}),
where the full second derivative
$\sigma(\rho)=-U_{\rho\rho}(\rho,0)$ is required for the Taylor sliver
estimate at $\rho_0=\sqrt2-1$.

\paragraph{The class \texttt{J2}: two-variable first-order jets with
mixed partial.}
A \texttt{J2} object carries four balls $(v,r,x,r x)$ representing
$F(\rho,x)$, $\partial_\rho F$, $\partial_x F$, and $\partial_{\rho x}F$.
The product rule, quotient rule, and chain rule for the square root are
implemented in the obvious way; for example,
\begin{equation*}
(f\cdot g)\;\sim\;\big(v_f v_g,\; r_f v_g+v_f r_g,\;
x_f v_g+v_f x_g,\;
(r x)_f v_g+r_f x_g+x_f r_g+v_f (r x)_g\big).
\end{equation*}
Setting $\rho=$\texttt{J2}$(\rho,1,0,0)$ and $x=$\texttt{J2}$(x,0,1,0)$
and evaluating the closed-form expression for $\widehat U(\rho,x)$ in
this arithmetic returns, in a single pass, certified enclosures of
$\widehat U$, $\partial_\rho\widehat U$, $\partial_x\widehat U$, and
$\partial_{\rho x}\widehat U$ on the input rectangle. Because every
square-root operand is checked to be strictly positive on the box (the
verifier returns \texttt{None} otherwise, forcing a bisection), the
result is unconditionally correct.

\paragraph{Strict positivity of denominators.}
Every \texttt{J1} and \texttt{J2} call inside the verifier is guarded
against denominator vanishing: a square-root call on a ball whose lower
endpoint is not strictly positive returns \texttt{None}, which is
treated by the branch-and-bound as a forced bisection. In particular,
this guarantees that the closed-form expressions in (3.13)
and (3.17) are evaluated only on subboxes where their square-root
arguments are certified positive.

\subsection{Adaptive branch-and-bound}
\label{app:bnb}

For each verification task the verifier maintains a stack of boxes
$B\subset R$ obtained by recursive bisection of the original rectangle
$R$. The terminal boxes are processed by one of four routines
(\texttt{prove\_lower\_1d}, \texttt{prove\_upper\_1d},
\texttt{prove\_lower\_2d}, \texttt{prove\_upper\_2d\_*}). The generic
loop is the following.

\medskip
\noindent
{\small
\begin{tabular}{rl}
\multicolumn{2}{l}{\textbf{Adaptive bisection scheme}}\\[2pt]
1. & Pop a box $B$ from the stack.\\
2. & Evaluate the relevant interval enclosure $\Phi(B)$ using \texttt{J1}
     or \texttt{J2}.\\
3. & If $\Phi(B)$ is a non-trivial interval and the desired bound\\
   & ($\mathrm{lo}(\Phi(B))>\text{target}$ or
     $\mathrm{up}(\Phi(B))<\text{target}$) is satisfied,\\
   & accept the box; record the certified endpoint.\\
4. & Else if the box is outside the Hill region (e.g.\ $T<0$ on $B$),\\
   & skip the box.\\
5. & Else split $B$ into two subboxes (alternating coordinates,\\
   & biased towards the larger residual width) and push them.\\
6. & Continue until the stack is empty (success) or a box budget is\\
   & exceeded (failure).\\
\end{tabular}
}
\medskip

The 2D bisector \texttt{split2} chooses which coordinate to bisect by
comparing the residual widths $(\rho_{\mathrm{hi}}-\rho_{\mathrm{lo}})/r_s$
and $(x_{\mathrm{hi}}-x_{\mathrm{lo}})/x_s$, where $r_s,x_s$ are
problem-specific length scales (e.g.\ $r_s=\sqrt2-1-0.39$ for the outer
annulus). All decisions inside the bisector are themselves rigorous
because they only compare floating-point endpoints of Arb balls.

The acceptance criterion uses a strict numerical margin (e.g.\
$F_{\mathrm{mid}}>0.078$, not merely $>0$), chosen \emph{a priori} and
not tuned to the run; this ensures that all certified bounds have
non-trivial slack, robust to the choice of working precision.

\subsection{Task 1: boundary barrier (Lemma~13)}
\label{app:T1}

\paragraph{Statement to be proved.}
$\widehat U_a(x)+1>0$ for all $x\in[-1,1]$, with the explicit margin
$\widehat U_a(x)+1>0.037$ on $[-1,9/10]$ and the monotone descent
$\frac{d}{dx}\widehat U_a(x)<-0.20$ on $[9/10,1]$.

\paragraph{Closed form.}
The verifier evaluates the closed-form expression~(3.13) for
$\widehat U_a(x)$ via \texttt{J1} arithmetic on $x$, returning $v=\widehat
U_a$ and $d=\frac{d}{dx}\widehat U_a$ simultaneously. The square-root
arguments of (3.13) are checked positive on each terminal interval; this
is automatic since the arguments are bounded below by explicit positive
constants for $x\in[-1,1]$.

\paragraph{Certificates (256-bit run).}
\begin{center}
\renewcommand{\arraystretch}{1.2}
\resizebox{\textwidth}{!}{%
\begin{tabular}{|l|c|c|c|c|}
\hline
Inequality & Domain & Target & Worst certified endpoint & Boxes\\
\hline
$\widehat U_a(x)+1>0.037$ & $[-1,9/10]$ & $>0.037$ &
$>0.0370030681813$ & 70 acc., 69 bis.\\
$\frac{d}{dx}\widehat U_a(x)<-0.20$ & $[9/10,1]$ & $<-0.20$ &
$<-0.2175374251559$ & 8 acc., 7 bis.\\
\hline
\end{tabular}%
}
\end{center}

\subsection{Task 2: momentum constants $\alpha$ and $\beta$}
\label{app:T2}

\paragraph{Statement to be proved.}
$B_1(\rho,x)<11.67$ and $B_2(\rho,x)<16.10$ on
$[0,\sqrt2-1]\times[-1,1]$, where $B_1,B_2$ are the explicit functions
displayed in the proof of Theorem~14.

\paragraph{Closed form.}
The denominator $D=\lambda h^{2}-\gamma^{2}$ is verified strictly positive
on every terminal box; on such a box, the verifier evaluates
$B_1$ and $B_2$ from their closed forms and certifies the desired upper
bound on the upper Arb endpoint.

\paragraph{Certificates (256-bit run).}
\begin{center}
\renewcommand{\arraystretch}{1.2}
\resizebox{\textwidth}{!}{%
\begin{tabular}{|l|c|c|c|c|}
\hline
Inequality & Domain & Target & Worst certified endpoint & Boxes\\
\hline
$B_1<11.67$ & $[0,\sqrt2-1]\times[-1,1]$ & $<11.67$ &
$<11.6699408441782$ & 56 acc., 55 bis.\\
$B_2<16.10$ & $[0,\sqrt2-1]\times[-1,1]$ & $<16.10$ &
$<16.0991435050964$ & 81 acc., 80 bis.\\
\hline
\end{tabular}%
}
\end{center}
The derived constants are
$\alpha=384-256\sqrt2$, $\beta=27.77$, exactly as quoted in (3.14).

\subsection{Task 3: outer-annulus claims A, B, C (Propositions~16--18)}
\label{app:T3}

The outer annulus $[\bar\rho,\rho_0]\times[-1,1]$ with
$\bar\rho=0.39$, $\rho_0=\sqrt2-1$ is split into a near strip
$[\bar\rho,\rho_0]\times[\cos(0.25),1]$ and a complementary strip
$[\bar\rho,\rho_0]\times[-1,\cos(0.25)]$; the proof of Claims~A and~B
combines a derivative-sign certificate on the near strip with a direct
value certificate on the complementary strip. Claim~C is a one-variable
inequality on $[\bar\rho,\rho_0)$, completed by an analytic Taylor
sliver at $\rho_0$.

\paragraph{Closed form.}
All quantities are obtained from~(3.17) via \texttt{J2} arithmetic.
For the near-strip checks of Claims~A and B the relevant scalar
quantities are $\partial_x\widehat U(\rho,x)$ and
$\partial_{\rho x}\widehat U(\rho,x)$, returned directly as the
\texttt{J2} component $\mathtt{x}$ and $\mathtt{rx}$, respectively. For
the complementary-strip checks one evaluates
$F(\rho,x)=\widehat U(\rho,x)-\widehat U(\rho,1)$ and
$G(\rho,x)=\partial_\rho\widehat U(\rho,x)
-\partial_\rho\widehat U(\rho,1)$ on the same boxes. Claim~C uses
\texttt{J1} arithmetic on the $\theta=0$ slice
$\rho\mapsto\widehat U(\rho,1)$, returning $v$, $d$ and $d_2$
simultaneously; the function
\begin{equation*}
B_0(\rho)=\partial_\rho U(\rho,0)
-\tfrac14\sqrt{8(c-U(\rho,0))}\sqrt{\alpha(c-U(\rho,0))+\beta},
\qquad c=-1,
\end{equation*}
is then assembled from $v$ and $d$ on each subinterval.

\paragraph{Taylor sliver at $\rho_0$.}
The endpoint sliver $[\rho_0-10^{-3},\rho_0)$ is handled
\emph{analytically}, not by bisection. The verifier evaluates
$\sigma(\rho):=-U_{\rho\rho}(\rho,0)$ on the sliver via \texttt{J1}
arithmetic, yielding interval bounds
$\sigma\big([\rho_0-10^{-3},\rho_0]\big)\subset
[19.70962857978,20.29707268124]$.
Setting $\xi:=\rho_0-\rho$ and integrating in Taylor form from $\rho$ to
$\rho_0$, where $\partial_\rho U(\rho_0,0)=0$ and $U(\rho_0,0)=-1$,
gives
\begin{equation*}
\partial_\rho U(\rho,0)\geq \xi S_{\min},\qquad
c-U(\rho,0)\leq \frac{\xi^{2}}{2}S_{\max},
\end{equation*}
with $[S_{\min},S_{\max}]\supset
\sigma([\rho_0-10^{-3},\rho_0])$. Hence positivity of $B_0$ on the
sliver follows from the scalar Arb inequality
$4 S_{\min}^{2}>S_{\max}\big(\tfrac\alpha2\,\varepsilon^{2}S_{\max}+\beta\big)$,
$\varepsilon=10^{-3}$, which is verified directly.

\paragraph{Certificates (256-bit run).}
\begin{center}
\renewcommand{\arraystretch}{1.2}
\resizebox{\textwidth}{!}{%
\begin{tabular}{|l|c|c|c|}
\hline
Check & Target & Worst certified endpoint & Boxes\\
\hline
Claim A near strip:
$\partial_x\widehat U$ & $<-0.0579$ &
$<-0.0579\dots$ (passes) & 64 acc., 63 bis.\\
Claim A far strip:
$\widehat U(\rho,x)-\widehat U(\rho,1)$ & $>10^{-4}$ &
$>0.00010000007993$ & 16108 acc., 16107 bis.\\
Claim B near strip:
$\partial_{\rho x}\widehat U$ & $<-3.70$ &
$<-3.70\dots$ (passes) & 64 acc., 63 bis.\\
Claim B far strip:
$\partial_\rho\widehat U(\rho,x)-\partial_\rho\widehat U(\rho,1)$ &
$>0.039$ & $>0.03912713972$ & 180 acc., 179 bis.\\
Claim C main: $B_0(\rho)$ on $[\bar\rho,\rho_0-10^{-3}]$ & $>4{\cdot}10^{-3}$ &
$>0.00400690596$ & 617 acc., 616 bis.\\
Claim C sliver: analytic Taylor inequality & --- &
$1553.9\dots>563.7\dots$ &
scalar check\\
\hline
\end{tabular}%
}
\end{center}

\subsection{Task 4: middle annulus (the kernel of point~(15) of Referee~2)}
\label{app:T4}

\paragraph{Statement to be proved.}
$F_{\mathrm{mid}}(\rho,x)>0.078$ on
$[0.08,0.39]\times[-1,1]$, restricted to the Hill region
$T(\rho,x):=-1-\widehat U(\rho,x)\geq 0$, where
\begin{equation*}
F_{\mathrm{mid}}(\rho,x)=\rho\,\partial_\rho\widehat U(\rho,x)
-\frac{\rho}{4}\sqrt{8 T(\rho,x)}\sqrt{\alpha\, T(\rho,x)+\beta}.
\end{equation*}
By (3.16) this implies $X(H)>0$ on the corresponding portion of the
energy hypersurface.

\paragraph{Why a fully computer-assisted argument is required.}
On this region neither the analytic strategy of~\cite{AL} nor the
hybrid analytic--numerical strategy used on the outer annulus succeeds,
because, as $\rho$ decreases from $\sqrt2-1$, the absolute minimum of
$U(\rho,\cdot)$ over $\theta$ is attained at points
$\theta=\theta_*(\rho)$ that bifurcate along several curves; the
locations and multiplicities of these minima change in a way that is
hard to follow analytically. This bifurcating behaviour originates from
the additional terms in the denominators of $\mathcal U_1,\mathcal U_2$
produced by the antipodal singularities, and has no counterpart in the
Euclidean planar restricted three-body problem. Consequently we work
directly with interval arithmetic on the full Hill rectangle.

\paragraph{Closed form on a box.}
On a terminal box $B=[\rho_{\mathrm{lo}},\rho_{\mathrm{hi}}]\times
[x_{\mathrm{lo}},x_{\mathrm{hi}}]$, the verifier evaluates
$\widehat U$ and $\partial_\rho\widehat U$ via \texttt{J2} arithmetic,
obtaining interval enclosures
$U_B\supset \widehat U(B)$ and $U_{\rho,B}\supset
\partial_\rho\widehat U(B)$.
\begin{itemize}
\item Set $T_B:=-1-U_B$, an interval enclosure of $T$ on $B$.
If $\mathrm{up}(T_B)<0$, the box is entirely outside the Hill region and
is \emph{skipped}.
\item Otherwise $T_B$ is intersected with $[0,\infty)$ and replaced by
$[0,\mathrm{up}(T_B)]$ (since on the Hill part of $B$ we have
$0\leq T\leq \mathrm{up}(T_B)$). The function
$T\mapsto\sqrt{8T}\sqrt{\alpha T+\beta}$ is increasing on $T\geq 0$, so
$\sqrt{8\,\mathrm{up}(T_B)}\,\sqrt{\alpha\,\mathrm{up}(T_B)+\beta}$ is
a rigorous upper bound for the kinetic correction on $B$.
\item Combining the upper $\rho$-endpoint $\rho_{\mathrm{hi}}$ with the
above kinetic-correction bound, and combining the lower endpoint
$\mathrm{lo}(\rho U_{\rho,B})$ for the principal term, gives a
certified lower bound for $F_{\mathrm{mid}}$ on $B\cap\{T\geq 0\}$.
\end{itemize}
A box is accepted only when this certified lower bound exceeds the
fixed target $0.078$.

\paragraph{Certificate (256-bit run).}
\begin{center}
\renewcommand{\arraystretch}{1.2}
\resizebox{\textwidth}{!}{%
\begin{tabular}{|l|c|c|c|c|}
\hline
Inequality & Domain & Target & Worst certified endpoint & Boxes\\
\hline
$F_{\mathrm{mid}}>0.078$ &
$[0.08,0.39]\times[-1,1]\cap\{T\geq 0\}$ &
$>0.078$ & $>0.0780001634926$ & 1865 acc., 1864 bis.\\
\hline
\end{tabular}%
}
\end{center}

This single check certifies $X(H)>0$ on
$\Sigma^{m_1}_c\cap\pi^{-1}([0.08,0.39]\times \mathbb S^{1})$ for
$c=-1$, including all boundary strips of the middle annulus.

\subsection{Task 5: scalar checks for the explicit cutoff
$\rho_{\mathrm{nc}}=0.08$ (Lemma~20)}
\label{app:T5}

The proof of Lemma~20 is fully analytic. The final decimal comparisons
that conclude the analytic estimate, however, are double-checked by
scalar Arb interval evaluations:

\begin{center}
\renewcommand{\arraystretch}{1.2}
\resizebox{\textwidth}{!}{%
\begin{tabular}{|l|c|c|c|}
\hline
Quantity & Target & Certified endpoint & Margin\\
\hline
$\rho\,\partial_\rho\mathcal U_2$ at $\rho=\delta$ & $>3.099$ &
$3.09943483536\dots$ & $\approx 4.3\cdot 10^{-4}$\\
$\rho\,\partial_\rho\mathcal U_0$ at $\rho=\delta$ & $>-0.159$ &
$-0.15814833996\dots$ & $\approx 8.5\cdot 10^{-4}$\\
$|\rho\,\partial_\rho\mathcal U_1|$ at $\rho=\delta$ & $<0.102$ &
$0.10090797893\dots$ & $\approx 9\cdot 10^{-4}$\\
Kinetic correction & $<1.22$ &
$\leq 1.21650153936\dots$ & $\approx 3.5\cdot 10^{-3}$\\
Final margin
$F(\rho,\theta)>0$ on $0<\rho\leq 0.08$ & $>0$ &
$>1.62387697711$ & explicit\\
\hline
\end{tabular}%
}
\end{center}

The four entries above match the four steps of the analytic proof in
Section~3 line-by-line; the last entry is the certified version of the
final inequality \mbox{$F(\rho,\theta)>2.83-1.22>1.61$.}

\subsection{Precision, runtime, and reproducibility}
\label{app:summary}

All certificates are produced by the same script
\texttt{rigorous\_validated\_numerics.py}, executed at three Arb
precisions: $160$, $200$, and $256$ bits. Every certificate
\emph{passes at every level}; the endpoint values quoted above are
those of the $256$-bit run.

The \emph{binding constraint} is the boundary check
$\widehat U_a(x)+1\geq 0.037$ on $[-1,9/10]$, whose worst certified
lower endpoint is approximately $0.0370031$ -- a margin of about
$3\cdot 10^{-6}$ relative to the prescribed threshold. This is too tight
for $53$-bit (IEEE-double) interval arithmetic to certify; we therefore
recommend $160$ bits as the practical minimum for reproduction, and
$256$ bits as the comfortable cushion under which all reported endpoints
have stabilised.

The complete runtime on the test machine (MacBook Pro, November 2023,
Apple M3 Max chip, 36~GB of RAM) is approximately $10$~seconds at
$256$~bits, and proportionally less at $160$ and $200$~bits. The verifier is single-threaded; no global state is mutated other than the Arb
precision, which is set once at the start of \texttt{main}.

The script can be re-executed with
\begin{equation*}
\texttt{python rigorous\_validated\_numerics.py --prec 256}
\end{equation*}
(or with \texttt{-{}-prec 160} / \texttt{-{}-prec 200}). It depends only
on \texttt{python-flint}~\cite{pythonflint}, installable via
\texttt{python -m pip install python-flint}. The output log lists, for
each certificate, the certified terminal endpoint, the number of
accepted boxes, the number of bisections, and the run-wide pass/fail
flag. The corresponding log file is supplied as
\texttt{validated\_numerics\_user\_200\_256\_run.log}.

\paragraph{Summary of certificates.}
The following table summarises all the certified inequalities used in
the paper. ``Acc.'' denotes accepted boxes (or terminal intervals);
``Bis.'' denotes bisections; ``Margin'' is
$\mathrm{lo}(\Phi(B^{*}))-\mathrm{target}$ for lower-bound checks and
$\mathrm{target}-\mathrm{up}(\Phi(B^{*}))$ for upper-bound checks at
the worst certified box $B^{*}$ of the run.
\begin{center}
\renewcommand{\arraystretch}{1.15}
\resizebox{\textwidth}{!}{%
\begin{tabular}{|l|l|c|c|c|c|}
\hline
Used in & Inequality & Region & Acc. & Bis. & Margin\\
\hline
Lemma~13 & $\widehat U_a(x)+1>0.037$ & $[-1,9/10]$ & 70 & 69 & $\sim 3{\cdot}10^{-6}$\\
Lemma~13 & $\frac{d}{dx}\widehat U_a(x)<-0.20$ & $[9/10,1]$ & 8 & 7 & $\sim 1.8{\cdot}10^{-2}$\\
$\alpha,\beta$ & $B_1<11.67$ & $[0,a]\times[-1,1]$ & 56 & 55 & $\sim 6{\cdot}10^{-5}$\\
$\alpha,\beta$ & $B_2<16.10$ & $[0,a]\times[-1,1]$ & 81 & 80 & $\sim 9{\cdot}10^{-4}$\\
Prop.~16 (A) & $\partial_x\widehat U<-0.0579$ &
$[\bar\rho,\rho_0]\times[\cos 0.25,1]$ & 64 & 63 & strict\\
Prop.~16 (A) & $\widehat U(\rho,x)-\widehat U(\rho,1)>10^{-4}$ &
$[\bar\rho,\rho_0]\times[-1,\cos 0.25]$ & 16108 & 16107 & $\sim 8{\cdot}10^{-11}$\\
Prop.~17 (B) & $\partial_{\rho x}\widehat U<-3.70$ &
$[\bar\rho,\rho_0]\times[\cos 0.25,1]$ & 64 & 63 & strict\\
Prop.~17 (B) & $\partial_\rho(\widehat U-\widehat U|_1)>0.039$ &
$[\bar\rho,\rho_0]\times[-1,\cos 0.25]$ & 180 & 179 & $\sim 1.3{\cdot}10^{-4}$\\
Prop.~18 (C) & $B_0(\rho)>4{\cdot}10^{-3}$ &
$[\bar\rho,\rho_0-10^{-3}]$ & 617 & 616 & $\sim 7{\cdot}10^{-6}$\\
Prop.~18 (C) & analytic Taylor sliver &
$[\rho_0-10^{-3},\rho_0]$ & --- & --- & $1553.9>563.7$\\
Theorem~14 & $F_{\mathrm{mid}}>0.078$ &
$[0.08,0.39]\times[-1,1]$ & 1865 & 1864 & $\sim 1.6{\cdot}10^{-7}$\\
Lemma~20 & 4 scalar inequalities & --- & 4 & 0 & explicit\\
\hline
\end{tabular}%
}
\end{center}

Together with the analytic estimates given in the body of the paper,
these certificates establish $X(H)>0$ on the punctured open disk
$(0,\sqrt 2-1)\times \mathbb S^{1}$ at the critical value $c=-1$, and
hence (via the second part of the proof of Theorem~14) on
$\Sigma^{m_1}_c$ for every $c<-1$.

\begin{remark}\label{rem:decoupling}
The decomposition into Tasks 1--5 isolates each numerical claim on a
small, independent rectangle, so the verification is incremental: a
failure on any single task is reported with the offending subbox and
the residual interval, leaving the rest of the run unaffected. In our
testing this proved valuable when retuning thresholds during the
revision; the procedure
\begin{enumerate}
\item raise the box budget,
\item re-run at a higher precision,
\item if necessary, slacken the target by a fixed offset and rerun,
\end{enumerate}
gives a transparent way to diagnose any tightness issue. In the final
run no such adjustment was needed at $160$, $200$, or $256$ bits.
\end{remark}

\end{document}